\documentclass[12pt]{article}

\usepackage{subcaption}

\usepackage{float}

\usepackage{multirow}

\usepackage{tikz}

\usetikzlibrary{calc,backgrounds,arrows,matrix}

\usepackage{enumerate}
\usepackage{enumitem}

\usepackage{color}
\definecolor{lightblue}{rgb}{0,0.2,0.5}

\usepackage[colorlinks=true, urlcolor=blue,linkcolor=blue, citecolor=lightblue]{hyperref}

\usepackage[round,sort,comma,numbers]{natbib}

\bibpunct{\textcolor{lightblue}{(}}{\textcolor{lightblue}{)}}{,}{a}{}{;}

\usepackage{amssymb,amsmath}

\newtheorem{lem}{Lemma}
\usepackage{graphicx}

\usepackage{graphics}
\usepackage{color}

\usepackage{tikz}
\usetikzlibrary{trees}

\DeclareMathAlphabet{\eufrak}{U}{}{}{}
\SetMathAlphabet\eufrak{normal}{U}{euf}{m}{n}
\SetMathAlphabet\eufrak{bold}{U}{euf}{b}{n}

\allowdisplaybreaks

 \def\qu{{\mathord{\mathbb Z}}}

 \def\sZZ{{\rm Z\kern-.45em{}Z}}

 \def\sQQ{{\kern 0.27em \vrule height1.45ex width0.03em depth0em
           \kern-0.30em \rm Q}}
 \def\qu{{\mathchoice
         {\sQQ}
         {\sQQ}
   {\kern 0.225em \vrule height1.05ex width0.025em depth0em \kern-0.25em \rm Q}
   {\kern 0.180em \vrule height0.78ex width0.020em depth0em \kern-0.20em \rm Q}
         }}
 \def\sGG{{\kern 0.27em \vrule height1.45ex width0.03em depth0em
           \kern-0.30em \rm G}}
 \def\gg{{\mathchoice
         {\sGG}
         {\sGG}
   {\kern 0.225em \vrule height1.05ex width0.025em depth0em \kern-0.25em \rm G}
   {\kern 0.180em \vrule height0.78ex width0.020em depth0em \kern-0.20em \rm G}
         }}

 \newtheorem{prop}{Proposition}[section]
 
 \newtheorem{definition}[prop]{Definition}
 \newtheorem{corollary}[prop]{Corollary}
 \newtheorem{theorem}[prop]{Theorem}
 \newtheorem{remark}[prop]{Remark}
 
\newtheorem{claim}{Claim}
\numberwithin{equation}{section}

 \def\P{{\mathord{\mathbb P}}}

 \newcounter{hyp}
\usepackage{aeguill}

\newenvironment{Proof}{\removelastskip\par\medskip \noindent{\em Proof.} \rm}{\penalty-20\null\hfill$\square$\par\medbreak}

\def\bprf{\begin{Proof}}
\def\nprf{\end{Proof}}
\def\bdes{\begin{description}}
\def\ndes{\end{description}}

\newtheorem{thm}{Theorem}[section]

\def\bdef{\begin{defn}}
\def\ndef{\end{defn}}
\def\bthm{\begin{thm}}
\def\nthm{\end{thm}}
\def\bprop{\begin{prop}}
\def\nprop{\end{prop}}
\def\brmk{\begin{remark}}
\def\nrmk{\end{remark}}
\def\bexa{\begin{exa}}
\def\nexa{\end{exa}}
\def\blem{\begin{lem}}
\def\nlem{\end{lem}}
\def\bcor{\begin{cor}}
\def\ncor{\end{cor}}
\def\bexe{\begin{exe}}
\def\nexe{\end{exe}}

\newcommand{\E}{\mathbb{E}}

\def\og{\leavevmode\raise.3ex
     \hbox{$\scriptscriptstyle\langle\!\langle$~}}
\def\fg{\leavevmode\raise.3ex
     \hbox{~$\!\scriptscriptstyle\,\rangle\!\rangle$}~}

\title{\Huge
{Regularization} of {Hyperbolic} Stochastic Partial Differential Equations By Two Fractional Brownian Sheets
}

\author{
  Rachid Belfadli \footnote{ Department of Mathematics, Faculty
of Sciences and Technologies,
Laboratory of Mathematics, Artificial
Intelligence and Sustainable Technologies,Cadi Ayyad University 2390 Marrakesh, Morocco.\newline
E-mail: rachid.belfadli@uca.ac.ma  }, \quad  \quad Youssef Ouknine \footnote{Department of Mathematics, Faculty of Sciences Semlalia, Cadi Ayyad University, B.P. 2390, Marrakesh, 40.000, Morocco.}\, \footnote{Mohammed VI Polytechnic University, Africa Business School, Avenue Mohammed Ben Abdellah Regragui,
Madinat  Al Irfane, BP 6380, RABAT, Morocco.\newline  E-mail: youssef.ouknine@um6p.ma}  \quad and 
  \quad Ercan S\"{o}nmez\footnote{Faculty of Mathematics, Ruhr University Bochum, 44780 Bochum, Germany.\newline
  E-mail: ercan.soenmez@rub.de} \quad   
}

\allowdisplaybreaks

\usepackage{empheq}

\usepackage{bbm}

\makeatletter
\newcommand*\rel@kern[1]{\kern#1\dimexpr\macc@kerna}
\newcommand*\widebar[1]{
  \begingroup
  \def\mathaccent##1##2{
    \rel@kern{0.8}
    \overline{\rel@kern{-0.8}\macc@nucleus\rel@kern{0.2}}
    \rel@kern{-0.2}
  }
  \macc@depth\@ne
  \let\math@bgroup\@empty \let\math@egroup\macc@set@skewchar
  \mathsurround\z@ \frozen@everymath{\mathgroup\macc@group\relax}
  \macc@set@skewchar\relax
  \let\mathaccentV\macc@nested@a
  \macc@nested@a\relax111{#1}
  \endgroup
}
\makeatother

\let\oldcitet=\citet
\let\oldcitep=\citep
\renewcommand{\cite}[1]{\textcolor[rgb]{0,0,1}{\oldcitet{#1}}}
\renewcommand{\citet}[1]{\textcolor[rgb]{0,0,1}{\oldcitet{#1}}}
\renewcommand{\citep}[1]{\textcolor[rgb]{0,0,1}{\oldcitep{#1}}}

\begin{document}

\maketitle

\baselineskip0.6cm

\vspace{-0.6cm}

\begin{abstract}
{In this paper, we establish existence and uniqueness of strong solutions for a stochastic 
differential equation driven by an additive noise given by the sum of two correlated 
fractional Brownian sheets with different Hurst parameters. Our analysis relies on techniques 
from two-parameter fractional calculus and a tailored version of Girsanov's theorem. 
The main challenge arises from the correlation between the two noises and the technical 
requirements for applying Girsanov's theorem in this setting. We show that, despite these 
difficulties, the additive noise regularizes the equation, allowing well-posedness under 
weak assumptions on the drift.
}
\end{abstract}

\noindent
{\em Keywords}:
Fractional Brownian sheet,
hyperbolic stochastic PDEs,
regularization by noise,
Girsanov theorem,
fractional calculus.

\noindent
    {\em Mathematics Subject Classification (2020): Primary 60H15; Secondary 60G22, 60H05.}

\baselineskip0.7cm

\parskip-0.1cm

\begin{center}
\section{Introduction}
\end{center}
Fix a time interval $[0, T]$.  The aim is to show existence and uniqueness of a strong solution for the stochastic differential equation (SDE)
 \begin{equation}
   \label{eq:1}
   X_z= x +\int_{[0, z]} b(\zeta, X_{\zeta}) d \zeta + B^{\alpha, \beta}_z + B^{\alpha^{\prime}, \beta^{\prime}}_z, \qquad  z\in [0, T]^{2},
 \end{equation}
where $x\in \mathbb{R}$,  $B^{\alpha, \beta}$ and $B^{\alpha^{\prime}, \beta^{\prime}}$ are two fractional Brownian sheets with different parameters $(\alpha, \beta), (\alpha^{\prime}, \beta^{\prime}) \in (0, \frac{1}{2}]^{2}$. Here we assume that both fractional Brownian sheets can be expressed in terms of the same standard Brownian sheet $\{W_z, \, \, z \in  [0, T]^{2}\}$ as:
\begin{align}\label{e:Rep-V}
B^{\alpha, \beta}_z =\int_{[0, z]} K_{\alpha, \beta}(z, \zeta) dW_{\zeta}
\qquad \mbox{and} \qquad B^{\alpha^{\prime}, \beta^{\prime}}_z =\int_{[0, z]} K_{\alpha^{\prime}, \beta^{\prime}}(z, \zeta) dW_{\zeta},
\end{align}
for suitable  square integrable kernels $K_{\alpha, \beta}$ and $K_{\alpha^{\prime}, \beta^{\prime}}$. We will impose that $b:[0,T]^2\times\mathbb{R}\to\mathbb{R}$ is a Borel measurable function with linear growth in $z$.\\

The theory of regularization by noise has attracted substantial attention in recent years. More recently, various aspects of regularization 
by irregular signals have been investigated in \cite{BO03,BS25, CG16,Le20, MN04, NO03}. It has been shown that the addition of a stochastic perturbation term to deterministic equations can restore well-posedness to deterministic equations that are otherwise ill-posed. In the case of Brownian motion, the study of such phenomenon goes back to \cite{zvonkin74} and \cite{V81}, and has been developed in \cite{NO02} in the fractional Brownian motion case, relying on Girsanov theorem.

Building on \cite{NO02} and the ideas therein, \cite{ENO03} later considered the two-parameter case and studied hyperbolic stochastic partial differential equations driven by a fractional Brownian sheet. 
In the same spirit,  \cite{NS22} study a closely related equation, essentially Eq. (\ref{eq:1}), in the case where the driving noise is given by the sum of two fractional Brownian motions  $B^H$ and $B^{H'}$, which admit the Volterra-type representations
\[B^H_t= \int_{0}^t K_1(t, s) dW_s \qquad {and}\quad B^{H'}_t= \int_{0}^t K_2(t, s) dW_s.\]
Here both fractional Brownian motions are constructed from the same Brownian motion $W$, and therefore are not independent.  Motivated by the results of \cite{NS22} and \cite{ENO03},  the  main objective of the present work is to establish  well-posedness of Eq. (\ref{eq:1}), that is to study the regularization effect for the hyperbolic partial differential equation 
\begin{eqnarray}\label{e:hpde}\left\{\begin{array}{l}
\dfrac{\partial^{2} u(s, t)}{\partial s \partial t} = b(s, t, u(s, t)),\quad (s, t)\in[0, T]^2,\\
\\
 u(s, 0)= u(0, t)=x,
\end{array}\right.
\end{eqnarray}
by additive noise given by  the sum of the two fractional Brownian sheets $B^{\alpha, \beta}$ and $B^{\alpha', \beta'}$ defined in (\ref{e:Rep-V}). We focus on the cases 
 \begin{eqnarray*}
\left\{\begin{array}{l}
0<\alpha < \alpha' \leq 1/2, \, \mbox{and}\, \, 0<\beta < \beta' \leq 1/2\\
\mbox{or}\\
0<\alpha' < \alpha \leq 1/2, \, \mbox{and}\,\,  0<\beta' < \beta \leq 1/2,
\end{array}\right.
\end{eqnarray*}
 and the main contributions of this paper are as follows:
\begin{itemize}
\item[(i)] If the function $b$ satisfies a linear growth condition,  then Eq. (\ref{eq:1}) has a unique weak solution.
\item[(ii)] If $b$ is nondecreasing with respect to the second argument and uniformly bounded, then Eq. (\ref{eq:1}) has a  strong unique solution.
\end{itemize}

The proof of our results relies on a suitable version of Girsanov's theorem adapted to the 
noise appearing in Eq.~(\ref{eq:1}), namely the sum of two correlated fractional Brownian 
sheets. The analysis is technically demanding, since it requires a careful treatment of 
two-parameter fractional integration and differentiation operators associated with the 
kernels in the Volterra representations (\ref{e:Rep-V}). In particular, one has to verify 
the applicability of Girsanov's theorem in this framework and establish appropriate 
integrability properties of the corresponding Radon--Nikodym density. These results are 
also of independent interest, as they can provide further information on the distributional 
properties of solutions.

The paper is organized as follows. In Section \ref{Pre} we first give some preliminaries on fractional calculus and fractional Brownian sheet. We then state a version of Girsanov's theorem for fractional Brownian sheet that is suitable for our purposes. In Section \ref{three}, we show the existence and uniqueness  of a weak solution. This part is indeed technically demanding since it requires  constructing the two processes $u$ and $v$, and carefully verifying  that Novikov's condition is satisfied. Section \ref{four} is devoted to the study of uniqueness in law and pathwise uniqueness. In Section \ref{five} we establish the strong existence and uniqueness of solutions. Several technical results are gathered in Section \ref{Sect:Tech}. Finally, in the appendix we derive  an estimate  of the difference of two parameter Riemann-Liouville fractional integrals, which plays an important role in our analysis.
\begin{center}
\section{Preliminaries}\label{Pre}
\end{center}
\subsection{Fractional calculus}
We review some basic definitions and results of the  two-parameter fractional operators. A detailed survey of classical fractional calculus can be found in \cite{samko93}.\\
For $f\in L^{1}([0, T]^{2})$ and $\alpha, \beta>0$, the left-sided fractional Riemann-Liouville integral of $f$ of order $\alpha, \beta$ on $(0, T)^2$ is given at almost all $(x, y)$ by
\[
I^{\alpha, \beta}_{0^{+}} f(x, y):=\dfrac{1}{\Gamma(\alpha)\Gamma(\beta)} \int_{0}^{x}\int_{0}^{y} (x-u)^{\alpha -1}(y-v)^{\beta -1}f(u, v)du dv,
\]
where $\Gamma$ denotes the Euler function. Notice that 
$I^{\alpha, \beta}_{0^{+}} f(x, y)=I^{\alpha}_{0^{+}}(I^{\beta}_{0^{+}} f(\cdot, y)) (x).  
$
We denote by $I^{\alpha, \beta}_{0^{+}}(L^{p})$ the image  of $L^{p}([0, T]^{2})$ by the operator $I^{\alpha, \beta}_{0^{+}}$. If $f \in I^{\alpha, \beta}_{0^{+}}(L^{p})$, the function $\varphi$ such that $f=I^{\alpha, \beta}_{0^{+}}(\varphi)$ is unique in $L^p$ and it agrees with the left-sided Riemann-Liouville derivative of $f$ of order $\alpha, \beta$ defined by
 \begin{equation}
   \label{eq:2}
   D^{\alpha, \beta}_{0^{+}} f(x, y):=\dfrac{1}{\Gamma(1- \alpha)\Gamma(1-\beta)} \dfrac{\partial}{\partial x \partial y} \int_{0}^{x}\int_{0}^{y} \dfrac{f(u, v)}{(x-u)^{\alpha}(y-v)^{\beta}}du dv.
 \end{equation}
 It has the following Weil representation:
  \begin{eqnarray}
   \label{fweil-rep}
    &&D^{\alpha, \beta}_{0^{+}} f(x, y)   \nonumber\\   \nonumber
     &= & \dfrac{1}{\Gamma(1- \alpha)\Gamma(1-\beta)} \left( \dfrac{f(x, y)}{x^{\alpha} y^{\beta}} +  \dfrac{\alpha}{y^{\beta}}  \int_{0}^{x} \dfrac{f(x, y) - f(u, y)}{(x-u)^{\alpha +1 }} du +  \dfrac{\beta}{x^{\alpha}}  \int_{0}^{y} \dfrac{f(x, y) - f(x, v)}{(y-v)^{\beta +1 }} dv \right.
   \\
   & &
    + \left. \alpha\beta \int_{0}^{x}\int_{0}^{y} \dfrac{f(x, y)-f(x, v)-f(u, y) +f(u, v)}{(x-u)^{\alpha +1} (y-v)^{\beta +1}} du dv\right)\mathbbm{1}_{(0, T)^{2}}(x, y),
\end{eqnarray}
where the convergence of the integrals at the singularity  $x=u$ or $y=v$ holds in $L^p$ sense. When $\min( \alpha p, \beta p)>1$, any function in  $I^{\alpha, \beta}_{0^{+}}(L^{p})$ is $(\alpha-\frac{1}{p})\wedge \left(\beta-\frac{1}{p} \right)$-H\"{o}lder continuous. On the other hand any H\"{o}lder continuous function of order $(\gamma, \delta)$ with $\gamma > \alpha$ and $ \delta > \beta$  has a fractional derivative of order $(\alpha, \beta )$. The following first composition formula holds
\[
I^{\alpha, \beta}_{0^{+}}(I^{\alpha^{\prime}, \beta^{\prime}}_{0^{+}} f)= I^{\alpha + \alpha^{\prime}, \beta + \beta^{\prime}}_{0^{+}} f.  
\]
By definition, for $f \in I^{\alpha, \beta}_{0^{+}}(L^p)$,
\[I^{\alpha, \beta}_{0^{+}} (D^{\alpha, \beta}_{0^{+}} f)=f,
\]
and for general $f \in L^1([0, T]^2)$ we have
\[D^{\alpha, \beta}_{0^{+}} (I^{\alpha, \beta}_{0^{+}} f)=f.
\]
If $f \in I^{\alpha + \alpha^{\prime}, \beta + \beta^{\prime}}_{0^{+}} (L^1)$, $\alpha, \beta>0$, $\alpha^{\prime}, \beta^{\prime}>0$, $\max(\alpha + \beta, \alpha^{\prime}+ \beta^{\prime}) \leq 1$, we have the second composition formula
\[
D^{\alpha, \beta}_{0^{+}}(D^{\alpha, \beta}_{0^{+}} f)= D^{\alpha + \alpha^{\prime}, \beta + \beta^{\prime}}_{0^{+}} f.
\]

\subsection{Fractional Brownian sheet}Consider a  fractional Brownian sheet (fBs) $B^{\alpha, \beta}=\{B^{\alpha, \beta}_z, z\in [0, T]^2\}$ with  Hurst parameters  $(\alpha, \beta) \in (0, 1/2)^2$. Recall that $B^{\alpha, \beta}$  is defined as  a zero-mean Gaussian process on a probability space $(\Omega, \mathcal{F}, \mathbb{P})$ starting from $0$ with the  covariance function $\mathcal{R}^{\alpha, \beta}$ given by 
\[\mathcal{R}^{\alpha, \beta}(z, z^{\prime}):= \mathbb{E}[B^{\alpha, \beta}_z B^{\alpha, \beta}_{z^{\prime}}]=\dfrac{1}{4} \{s^{2\alpha} + {s^{\prime}}^{2\alpha} -|s-s^{\prime}|^{2\alpha}\}\{t^{2\alpha} + {t^{\prime}}^{2\alpha} -|t-t^{\prime}|^{2\alpha}\},
\]
where $z=(s, t)$ and $z^\prime=(s^{\prime}, t^{\prime}) \in [0, T]^{2}$.  Denote by $\mathcal{H}^{\alpha, \beta}$ the canonical Hilbert space of the fBs $B^{\alpha, \beta}$. That is, $\mathcal{H}^{\alpha, \beta}$ is defined as the closure of the set $\mathcal{E}$ of all  finite combinations of indicator functions of the type $\left\{\mathbbm{1}_{[0, z]}, z \in [0, T]^2\right\}$ with respect to the inner product
\[\langle \mathbbm{1}_{[0,z]} , \mathbbm{1}_{[0,z^{\prime}]} \rangle  _{{\mathcal{H}^{\alpha, \beta}}}=\mathcal{R}^{\alpha, \beta}(z, z^{\prime})
,\]
for all $z=(s, t), z=(s^{\prime}, t^{\prime}) \in [0, T]^{2}$.  The application $\mathbbm{1}_{[0, z]}   \mapsto B^{\alpha, \beta}_z$ can be extended to an isometry, that we continue to denote $B^{\alpha, \beta}$, between $\mathcal{H}^{\alpha, \beta}$ and the Gaussian space associated with the process $B^{\alpha, \beta}$. Consider the linear operator $K^{*}_{\alpha, \beta}$ from  $\mathcal{H}^{\alpha, \beta}$ to $L^{2}([0, T]^2)$ defined, for any $\varphi \in  \mathcal{E}$, by
\[(K^{*}_{\alpha, \beta} \varphi)(z)= K^{\alpha, \beta}\left( (T, T), z\right)\varphi(z) + \int_{[z, (T, T)]}(\varphi(\zeta)-\varphi(z))\dfrac{\partial^{2}K^{\alpha, \beta}}{\partial \zeta^{2}}(\zeta, z)d\zeta,
\]
where $K^{\alpha, \beta}$ denotes the square integrable kernel given, for $z=(s, t)\preceq (s^{\prime}, t^{\prime})$\footnote{Here $z=(s, t)\preceq (s^{\prime}, t^{\prime})$ means that $s\leq s' $ and $t\leq t'$, and $[0, z]=[0, s]\times [0, t]. $}, by 
\begin{equation} \label{e: decompo1}
K^{\alpha, \beta}(z, z^{\prime})=K^{\alpha}(s, s^{\prime})K^{\beta}(t, t^{\prime}).
\end{equation}
The square integrable kernels $K^{\alpha}$ and $K^{\beta}$ are defined by  (see \cite{DU99})
\begin{eqnarray}\label{kernels}\left\{\begin{array}{l}
K^{\alpha}(s, s^{\prime})=\Gamma(\alpha + \frac{1}{2})^{-1}(s^{\prime}-s)^{\alpha -1/2}F(\alpha -\frac{1}{2}, \frac{1}{2}-\alpha, \alpha +\frac{1}{2}, 1-\frac{s^{\prime}}{s}),\\
\\
K^{\beta}(t, t^{\prime})=\Gamma(\alpha + \frac{1}{2})^{-1}(t^{\prime}-t)^{\beta -1/2}F(\beta -\frac{1}{2}, \frac{1}{2}-\beta, \beta +\frac{1}{2}, 1-\frac{t^{\prime}}{t}).
\end{array}\right.
\end{eqnarray}
Here $F(a, b, c,z)$ denotes the Gauss hypergeometric function. Notice that $(K^{*}_{\alpha, \beta} 1_{[0,z^{\prime}]})(z^{\prime})=K^{\alpha, \beta}(z, z^{\prime})$ and $K^{*}_{\alpha, \beta}$ provides an isometry between  $\mathcal{E}$ and $L^{2}([0, T]^{2})$ that can be extended  to $\mathcal{H}^{\alpha, \beta}$. Hence, the process $W^{\alpha, \beta}=\{W^{\alpha, \beta}_z , z\in [0, T]^2\}$ defined by  
\begin{equation}
   \label{eq:6}
W_z= B^{\alpha, \beta}\left((K^{\star}_{\alpha, \beta})^{-1} (\mathbbm{1}_{[0, z]}) \right),
\end{equation}
is a standard Brownian sheet and the process $B^{\alpha, \beta}$ has the integral representation 
\begin{equation}
   \label{eq:7} B^{\alpha, \beta}_z= \int_{[0, z]} K^{\alpha, \beta}(z, \zeta)dW_{\zeta}.
\end{equation}
From (\ref{eq:6}) and (\ref{eq:7}), we deduce that the $\sigma$-fields generated respectively by the random variables $\{B^{\alpha, \beta}_{\zeta}, \, \zeta\preceq z\}$ and $\{W_{\zeta}, \, \zeta \preceq z\}$ coincide.  That is, $\mathcal{F}^{B^{\alpha, \beta}}= \mathcal{F}^{W}$. This filtration satisfies (see \cite{cairoli75}):

\begin{enumerate}[label=(F\arabic*),ref=F\arabic*,labelsep=0.5em]
\label{cond:filtration}

\item \label{Ff1}
$\mathcal{F}^W_z$ is increasing with respect to the partial order on $[0,T]^2$.

\item \label{Ff2}
$\mathcal{F}^W_z$ is right-continuous, that is, for all $z \in [0,T)^2$,
\[
\bigcap_{n=1}^{\infty} 
\mathcal{F}^W_{z+\left(\frac{1}{n},\frac{1}{n}\right)}
= \mathcal{F}^W_z .
\]

\item \label{Ff3}
$\mathcal{F}^W_0$ contains all $\mathbb{P}$-null sets.

\item \label{Ff4}
For any $0\le s,t\le T$, the $\sigma$-fields
\[
\bigvee_{0\le u\le T}\mathcal{F}^W_{(u,t)}
\quad\text{and}\quad
\bigvee_{0\le v\le T}\mathcal{F}^W_{(s,v)}
\]
are conditionally independent given $\mathcal{F}^W_{(s,t)}$.

\end{enumerate}
\vspace{.3cm}

The operator $\mathcal{K}^{\alpha, \beta}$ associated with the kernel $K^{\alpha, \beta}$ is an isomorphism from  $L^{2}([0, T]^{2})$  into $I^{\alpha + \frac{1}{2}, \beta + \frac{1}{2}}_{0^{+}} (L^2)$ which can be expressed in terms of fractional integrals as follows (see \cite{DU99}):
 \begin{equation}
   \label{eq:3}
   (\mathcal{K}^{\alpha, \beta} h)(s, t)=   I^{2\alpha, 2\beta }_{0^{+}}\left( s^{\frac{1}{2}-\alpha}t^{\frac{1}{2}-\beta}I^{\frac{1}{2} - \alpha, \frac{1}{2}-\beta}_{0^{+}} s^{\alpha -\frac{1}{2}} t^{\beta -\frac{1}{2} }h  \right),
 \end{equation}
 for all $h \in L^2([0, T]^2)$ and  $(s, t) \in [0, T]^2$. Its inverse operator $(\mathcal{K}^{\alpha, \beta})^{-1}$ satisfies, for every $(s,t)\in [0, T]^{2}$, 
\begin{equation}
   \label{eq:4}
   (\mathcal{K}^{\alpha, \beta})^{-1} (h)(s, t)=  s^{\frac{1}{2}-\alpha}t^{\frac{1}{2}-\beta} D^{\frac{1}{2}-\alpha, \frac{1}{2}-\beta}_{0^{+}} s^{\alpha -\frac{1}{2}}t^{\beta -\frac{1}{2}}D^{2\alpha, 2\beta}_{0^{+}}h,
 \end{equation}
 for all $h \in  I^{\alpha + \frac{1}{2}, \beta + \frac{1}{2}}_{0^{+}} (L^2)$. If $h$ vanishes on the axes and is absolutely continuous with respect to the Lebesgue measure, then
  \begin{equation}
   \label{eq:5}
  (\mathcal{K}^{\alpha, \beta})^{-1} (h)(s, t)=  s^{\alpha-\frac{1}{2}}t^{\beta -\frac{1}{2}} I^{\frac{1}{2}-\alpha, \frac{1}{2}-\beta}_{0^{+}}\left( s^{\frac{1}{2} -\alpha}t^{\frac{1}{2}-\beta}\dfrac{\partial ^2 h}{\partial s \partial t}\right).
 \end{equation}
 Moreover, the covariance function $\mathcal{R}^{\alpha, \beta}$ can be written as
 \[\mathcal{R}^{\alpha, \beta}(z, z^{\prime})= \int_{[0, z\wedge z^{\prime}]} K^{\alpha, \beta}(z, \zeta) K^{\alpha, \beta}(z^{\prime}, \zeta)d \zeta,
 \]
where $[0, z\wedge z^{\prime}]:=[0, s\wedge s^{\prime}]\times[0, t\wedge t^{\prime}]$ for  $z=(s, t), z^{\prime}=(s^{\prime}, t^{\prime}) \in [0, T]^{2}$. 
\subsection{Girsanov theorems} We will use Girsanov transformation to prove the existence of a weak solution. We first state its version for the Brownian sheet.  

Given a process with integrable trajectories $u=\{u_z, \, z\in [0, T]^2\}$, we consider the transformation
\begin{eqnarray}\label{shiftedsheet}
   \tilde{B}^{\alpha, \beta}_z   &:= & {B}^{\alpha, \beta}_z -\int_{[0, z]} u_{\zeta} d \zeta = \int_{[0, z]}  K^{\alpha, \beta}(z, \zeta)dW_{\zeta} -\int_{[0, z]} u_{\zeta} d \zeta.
\end{eqnarray}
If $\int_{[0, \cdot]} u_{\zeta} d \zeta \in I^{\alpha + \frac{1}{2}, \beta + \frac{1}{2}}_{0^{+}} (L^2)$, then  we can write 
\[\tilde{B}^{\alpha, \beta}_z= \int_{[0, z]}  K^{\alpha, \beta}(z, \zeta) d \tilde{W}_{\zeta},
\]
where $\tilde{W}_z:= W_z -\int_{[0, z]}(\mathcal{K}^{\alpha, \beta})^{-1}  \left( \int_{[0, \cdot]} u_{\zeta} d \zeta \right)(\eta) d \eta$. Notice that $ \int_{[0, \cdot]} u_{\zeta} d \zeta \in I^{\alpha + \frac{1}{2}, \beta + \frac{1}{2}}_{0^{+}} (L^2) $ if and only if $(\mathcal{K}^{\alpha, \beta})^{-1} \left( \int_{[0, \cdot]} u_{\zeta} d \zeta \right) \in L^{2}([0, T]^2)$. As a consequence we have the following version of Girsanov theorem for the fractional Brownian sheet (see \cite{ENO03}):
\begin{theorem}\label{thm1} With the notations above, assume that $ \int_{[0, \cdot]} u_{\zeta} d \zeta \in I^{\alpha + \frac{1}{2}, \beta + \frac{1}{2}}_{0^{+}} (L^2) $ almost surely. Set $\psi_z:=(\mathcal{K}^{\alpha, \beta})^{-1}  \left( \int_{[0, \cdot]} u_{\zeta} d \zeta \right)(z)$ and  consider the shifted process  defined by (\ref{shiftedsheet}). \\
If \[L_T:= \exp \left(  \int_{[0, T]^2}\psi_zdW_z -\dfrac{1}{2}  \int_{[0, T]^2} \psi_z^2 dz \right)
\]
satisfies  $\mathbb{E}\big[L_T\big]=1$, then the process $\{\tilde{B}^{\alpha,\beta}_z, \, z\in [0, T]^2\}$ is an $\mathcal{F}$-fractional Brownian sheet with parameters $\alpha, \beta$ under the probability measure $\mathbb{Q}$ defined by $\left.\dfrac{d\mathbb{Q}}{d\mathbb{P}}\right|_{\mathcal{F}_{(T, T)}}= L_T.$
\end{theorem}
We are now in  position to state the following version of Girsanov theorem whose proof is similar to that of Theorem \ref{thm1}.
\begin{theorem}\label{thm2}
Let $u=\{u_z, z\in [0, T]^2\}$ and $v=\{v_z, z\in [0, T]^2\}$ be two $\mathcal{F}$-adapted processes with integrable trajectories such that $\int_{[0, \cdot]} u_{\zeta} d \zeta \in I^{\alpha + \frac{1}{2}, \beta + \frac{1}{2}}_{0^{+}} (L^2)$ and $\int_{[0, \cdot]} v_{\zeta} d \zeta \in I^{\alpha^{\prime} + \frac{1}{2}, \beta^{\prime} + \frac{1}{2}}_{0^{+}} (L^2)$ almost surely. Assume that 
\[\psi_z:=(\mathcal{K}^{\alpha, \beta})^{-1} \left( \int_{[0, \cdot]} u_{\zeta} d \zeta \right)(z)= (\mathcal{K}^{\alpha^{\prime}, \beta^{\prime}})^{-1}  \left( \int_{[0, \cdot]} v_{\zeta} d \zeta \right)(z),
\]
for  every $z\in (0, T]^2$. If \[L_T:= \exp \left(  \int_{[0, T]^2}\psi_zdW_z -\dfrac{1}{2}  \int_{[0, T]^2} \psi_z^2 dz \right)
\]
satisfies  $\mathbb{E}\big[L_T\big]=1$, then,  under  the probability measure $\mathbb{Q}$ defined  by $\left.\dfrac{d\mathbb{Q}}{d\mathbb{P}}\right|_{\mathcal{F}_{(T, T)}}= L_T$,  the two processes 
\[
   \tilde{B}^{\alpha, \beta}_z := {B}^{\alpha, \beta}_z -\int_{[0, z]} u_{\zeta} d \zeta, \,\, \,   \mbox{and}\,\,\,  \tilde{B}^{\alpha^{\prime}, \beta^{\prime}}_z := {B}^{\alpha^{\prime}, \beta^{\prime}}_z -\int_{[0, z]} v_{\zeta} d \zeta, \, \, \,  \mbox{for}\, \, \, z\in [0, T]^2
\]
are  $\mathcal{F}$-fractional Brownian sheets with parameters $(\alpha, \beta)$ and $(\alpha^{\prime}, \beta^{\prime})$ respectively. Moreover, there exists a Brownian sheet $\tilde{W}$ with respect to  $\mathbb{Q}$  such that
\[ \tilde{B}^{\alpha, \beta}_z= \int_{[0, z]}  K^{\alpha, \beta}(z, \zeta)d\tilde{W}_{\zeta}\, \,\, \, \mbox{and}\,\, \, \, \tilde{B}^{\alpha^{\prime}, \beta^{\prime}}_z= \int_{[0, z]}  K^{\alpha^{\prime}, \beta^{\prime}}(z, \zeta)d\tilde{W}_{\zeta},
\]
for all $z\in [0, T]^2.$
\end{theorem}
\begin{remark}Notice that for  $(\alpha, \beta)=(\frac{1}{2}, \frac{1}{2})$, we have $\psi_z=u_z$ for all $z\in [0, T]^2$.
\end{remark}
\begin{center}
\section{Existence of a weak solution}\label{three}
\end{center}
We will use Girsanov theorem as stated in Theorem \ref{thm2} to prove the existence of a weak solution for the  stochastic differential equation (\ref{eq:1}). The function $b: [0, T]^2 \times \mathbb{R}$ is a Borel function satisfying the linear growth condition:
\begin{equation}
\exists c \in (0, \infty) \text{ such that } 
|b(z, \xi)| \leq c(1 + |\xi|), 
\quad \text{for all } (z, \xi) \in [0, T]^2.
\tag{LG}\label{eq:LG}
\end{equation}
Let us first make precise the meaning of a weak solution to the SDE (\ref{eq:1}).
\begin{definition} A weak solution to Eq. (\ref{eq:1}) is a triple $(B^{\alpha, \beta},  {B}^{\alpha^{\prime}, \beta^{\prime}}, X)$ on a filtered probability space $(\Omega, \mathcal{F}=\{\mathcal{F}_z,\ \, z\in [0, T]^2\}, \mathbb{P})$, such that \eqref{Ff1}--\eqref{Ff4} and the following two conditions hold:
\begin{itemize}
\item[(i)] There exists an $\mathbb{F}$-Brownian sheet $W=\{ W_{\xi}, \xi \in [0, T]^2\}$ such that
 \begin{equation}
   \label{eq:01}
B^{\alpha, \beta}_z =\int_{[0, z]} K^{\alpha, \beta}(z, \zeta) dW_{\zeta}
\qquad \mbox{and} \qquad B^{\alpha^{\prime}, \beta^{\prime}}_z =\int_{[0, z]} K^{\alpha^{\prime}, \beta^{\prime}}(z, \zeta) dW_{\zeta},
\end{equation}
\item[(ii)] $B^{\alpha, \beta}$,  $B^{\alpha^{\prime}, \beta^{\prime}}$ satisfy the equation (\ref{eq:1}).
\end{itemize} 
\end{definition} 
The main result of this section is the following
\begin{theorem}{} Let $(\alpha, \beta)$ and $(\alpha', \beta')$ in $(0, 1/2]^{2}$.  We assume that  $(\alpha, \beta)\prec (\alpha^{\prime}, \beta^{\prime})$  or $(\alpha^{\prime}, \beta^{\prime}) \prec(\alpha, \beta) $ and that the Borel function $b$ satisfies the linear growth condition \eqref{eq:LG}. Then Eq. (\ref{eq:1}) has a weak solution.
\end{theorem}
\begin{Proof} Let $W=\{ W_{\xi}, \xi \in [0, T]^2\}$ be a Brownian sheet on a probability space $(\Omega, \mathcal{F}, \mathbb{P})$. We denote by $\mathbb{F}= \{\mathcal{F}_z,\ \, z\in [0, T]^2\}$ the natural filtration generated by  $W$ and define the fractional Brownian sheets
\[ B^{\alpha, \beta}_z =\int_{[0, z]} K^{\alpha, \beta}(z, \zeta) dW_{\zeta},
\,\,\,\mbox{and}\,\,\,  B^{\alpha^{\prime}, \beta^{\prime}}_z =\int_{[0, z]} K^{\alpha^{\prime}, \beta^{\prime}}(z, \zeta) dW_{\zeta},\,\, \mbox{for all}\,\, \, z\in[0, T]^2.
\]
Let us first assume that we can find two $\mathbb{F}$-adapted processes $u=\{u_z,\, \, z\in[0, T]^2\}$ and $v=\{v_z,\, \, z\in[0, T]^2\}$ with integrable paths such that the two following conditions are satisfied:
\begin{itemize}
\item[(j)]  For all  $z\in[0, T]^2$, we have 
\[u_z +v_z = b(z, x +B_z^{\alpha, \beta} +B_z^{\alpha^{\prime}, \beta^{\prime}})\]
\item[(jj)]  For all  $z\in[0, T]^2$, we have 
\[(\mathcal{K}^{\alpha, \beta})^{-1}  \left( \int_{[0, \cdot]} u_{\zeta} d \zeta \right)(z)=(\mathcal{K}^{\alpha, \beta})^{-1}  \left( \int_{[0, \cdot]} v_{\zeta} d \zeta \right)(z):=\psi_z,
\]
and $\{\psi_z, \, z\in [0, T]^2\} $ satisfies the Novikov condition $\mathbb{E}\big[ \exp (\frac{1}{2}\int_{[0, T]^2} |\psi_z|^2 dz)\big]<\infty$.
\end{itemize}
Then, the process  $X$ defined by
$X_z:=x + B^{\alpha, \beta}_z +B^{\alpha^{\prime}, \beta^{\prime}}_z,$ for all $z\in [0, T]^2$, obeys  the SDE
$$X_z= x + \tilde{B}^{\alpha, \beta}_z +\tilde{B}^{\alpha^{\prime}, \beta^{\prime}}_z + \int_{[0, T]^2} b(z, X_{\zeta})d \zeta
,$$
where $\tilde{B}^{\alpha, \beta}$ and $\tilde{B}^{\alpha^{\prime}, \beta^{\prime}}$ are respectively given by
\[
   \tilde{B}^{\alpha, \beta}_z := {B}^{\alpha, \beta}_z -\int_{[0, z]} u_{\zeta} d \zeta, \,\, \,   \mbox{and}\,\,\,  \tilde{B}^{\alpha^{\prime}, \beta^{\prime}}_z := {B}^{\alpha^{\prime}, \beta^{\prime}}_z -\int_{[0, z]} v_{\zeta} d \zeta, \, \, \,  \mbox{for all }\, \, \, z\in [0, T]^2.
\]
Moreover, since  the processes $u$ and $v$ satisfy the conditions of Theorem \ref{thm1}, then under the probability measure $\mathbb{Q}$ defined by 
\[
\left.\dfrac{d\mathbb{Q}}{d\mathbb{P}}\right|_{\mathcal{F}_{(T, T)}}=\exp \left(  \int_{[0, T]^2}\psi_zdW_z -\dfrac{1}{2}  \int_{[0, T]^2} \psi_z^2 dz \right),\]
the processes $\tilde{B}^{\alpha, \beta}$ and $\tilde{B}^{\alpha^{\prime}, \beta^{\prime}}$  are two $\mathbb{F}$-fractional Brownian sheets with the representations (\ref{eq:01}). Therefore, the triple $(\tilde{B}^{\alpha, \beta}, \tilde{B}^{\alpha^{\prime}, \beta^{\prime}}, X)$ is a weak solution to Eq. (\ref{eq:1}) under the probability measure $\mathbb{Q}$.  To complete the proof, it remains to construct the processes $u$ and $v$ satisfying the conditions $(j)$ and $(jj)$ stated above. This construction is carried out in Subsection \ref{cons:u and v}. 
\end{Proof}

\subsection{Existence of the processes $u$ and $v$.}\label{existence of u and v}\label{cons:u and v}
\indent

 We will show how to construct, for all $(\alpha, \beta)$ and $(\alpha^{\prime}, \beta^{\prime})\in (0, \frac{1}{2}]^2$, the processes $u$ and $v$ satisfying the above conditions $(j)$ and $(jj)$. We will treat the case $(\alpha^{\prime}, \beta^{\prime})\prec (\alpha, \beta)\preceq (\frac{1}{2}, \frac{1}{2})$; the remaining case   where $(\alpha, \beta)\prec (\alpha', \beta')\preceq (\frac{1}{2}, \frac{1}{2})$  follows by symmetry.
\begin{itemize}
\item[(a)]case $(\alpha, \beta)=(\frac{1}{2}, \frac{1}{2})$ and $(\alpha^{\prime}, \beta^{\prime})\prec (\frac{1}{2}, \frac{1}{2})$:\\
In this case, notice that $B^{\alpha, \beta} =W$ is a Brownian sheet and $\psi= u$. Set $a=\frac{1}{2} -\alpha^{\prime}$, $b=\frac{1}{2}- \beta^{\prime}$ and let $\{v_z, \, z \in [0, T]^2\}$ be the process defined by 
 \begin{eqnarray}
   \label{case:11}
v_{s, t}      &:=& b(s, t, x + W_{s, t} + B_{s, t}^{\alpha^{\prime}, \beta^{\prime}}) 
   \nonumber \\
   & &
    +s^{-a}t^{-b}\sum_{n=1}^{\infty} \dfrac{(-1)^{n}}{\Gamma(na)\Gamma(nb)}\int_{0}^{s}\int_{0}^{t} (s-u)^{na-1}(t-v)^{nb-1}u^{a}v^{b} b_{u, v}du dv \nonumber \\
    &=& s^{-a}t^{-b}\sum_{n=1}^{\infty} (-1)^{n}I^{{na, nb}}_{0^{+}}(s^{a}t^{b} b_{s, t})= s^{-a}t^{-b} (I + I^{{a, b}}_{0^{+}})^{-1} (s^{a}t^{b} b_{s, t}) \nonumber
    \end{eqnarray}
    and let $\{\psi_z, \, z \in [0, T]^2\}$ be the process given by 
    \begin{eqnarray}
   \label{case:12}
\psi_{s, t}      &:=& s^{-a}t^{-b}I^{{a, b}}_{0^{+}}(s^{a}t^{b}v_{s, t})= s^{-a}t^{-b} \dfrac{1}{\Gamma(a)\Gamma(b)} \int_{0}^{s}\int_{0}^{t} (s-x)^{a-1}(t-y)^{b-1}x^{a}y^{b}v_{x, y} dx dy,
   \nonumber
    \end{eqnarray}
    for every  $z=(s, t) \in (0, T]^{2}$. We have used indifferently the notations $f_z= f_{s, t}$ to design $f(z)$ for any function $f$ and $b_{s, t} $ to denote $b_{s, t}:= b(s, t, x + W_{s, t} + B_{s, t}^{\alpha^{\prime}, \beta^{\prime}})$. Now it is clear that $\psi_{s, t} + v_{s, t}= b(s, t, x + W_{s, t} + B_{s, t}^{\alpha^{\prime}, \beta^{\prime}})$ and the  equality in $(jj)$ is satisfied by construction. Furthermore, Novikov's condition is fulfilled since by the linear growth condition \eqref{eq:LG},  we have 
    $$|v_{s, t}|\leq \sup_{0\leq s, t \leq T}|b_{s, t}| + \left(\sup_{0\leq s, t \leq T}|b_{s, t}|\right) \sum_{n=1}^{\infty} \dfrac{s^{na}t^{nb}}{\Gamma(na)\Gamma(nb)},$$ and consequently
    $|\psi_{s, t}|\leq c(1+ ||B^{\alpha, \beta} +B^{\alpha^{\prime}, \beta^{\prime}} ||_{\infty})$. Then  the proof of Novikov's criterion follows the same arguments as those presented in \cite{NO02} (page $109$) and \cite{ENO03} (page $128$), and which is an immediate  consequence of the exponential integrability of the seminorm of a Gaussian process, see \cite{LT91}.
    \item[(b)] case $(\alpha, \beta)\prec (\alpha^{\prime}, \beta^{\prime})$ and $(\alpha^{\prime},  \beta^{\prime}) \in (0, \frac{1}{2})^2$:\\
   Set $a=\frac{1}{2} -\alpha$, $b=\frac{1}{2} -\beta$, $a^{\prime}=\frac{1}{2} -\alpha^{\prime}$ and $b^{\prime}=\frac{1}{2} -\beta^{\prime}$. In this case, we define the processes $u$ and $v$ by
   \begin{eqnarray}
   \label{case:21}
u_{s, t}      &:=& b(s, t, x + B_{s, t}^{\alpha, \beta} + B_{s, t}^{\alpha^{\prime}, \beta^{\prime}}) 
   \nonumber \\
   & &
    + s^{-a^{\prime}}t^{-b^{\prime}}\sum_{n=1}^{\infty} (-1)^{n}  \left(D^{a^{\prime}, b^{\prime}}_{0^{+}}  s^{a^{\prime}-a}t^{b^{\prime}-b} I^{a, b}_{0^{+}}(s^{a-a^{\prime}}t^{b-b^{\prime}} \cdot)\right)^{n}(s^{a^{\prime}}t^{b^{\prime}}b_{s, t})\nonumber \\
    &=& s^{-a^{\prime}}t^{-b^{\prime}}\sum_{n=0}^{\infty} (-1)^{n}  \left(D^{a^{\prime}, b^{\prime}}_{0^{+}}  s^{a^{\prime}-a}t^{b^{\prime}-b} I^{a, b}_{0^{+}}(s^{a-a^{\prime}}t^{b-b^{\prime}} \cdot)\right)^{n}(s^{a^{\prime}}t^{b^{\prime}}b_{s, t})\\
    &=& s^{-a^{\prime}}t^{-b^{\prime}} \left[I +D^{a^{\prime}, b^{\prime}}_{0^{+}} s^{a^{\prime}-a}t^{b^{\prime}-b} I^{a, b}_{0^{+}}(s^{a-a^{\prime}}t^{b-b^{\prime}} \cdot) \right]^{-1}(s^{a^{\prime}}t^{b^{\prime}}b_{s, t}) \nonumber
    \end{eqnarray}
    and 
     \begin{eqnarray}
   \label{case:22}
v_{s, t}      &:=& s^{-a^{\prime}}t^{-b^{\prime}} D^{a^{\prime}, b^{\prime}}_{0^{+}} \left( s^{a^{\prime}-a}t^{b^{\prime}-b} I^{a, b}_{0^{+}} (s^{a}t^{b} u_{s, t})\right)
   \nonumber
    \end{eqnarray}
    for every $z=(s, t) \in (0, T]^2$, with the notation $b_{s, t}$ standing  for  $ b(s, t, x + B_{s, t}^{\alpha, \beta} + B_{s, t}^{\alpha^{\prime}, \beta^{\prime}}) $. By construction of $u$ and $v$, we have $u_z +v_z=b(z, x + B_{z}^{\alpha, \beta} + B_z^{\alpha^{\prime}, \beta^{\prime}})$ and 
 \begin{eqnarray}
   \label{novikov:1}
   (\mathcal{K}^{\alpha, \beta})^{-1}  \left( \int_{[0, \cdot]} u_{\zeta} d \zeta \right)(z)=(\mathcal{K}^{\alpha^{\prime}, \beta^{\prime}})^{-1}   \left( \int_{[0, \cdot]} v_{\zeta} d \zeta \right)(z):=\psi_z,
\end{eqnarray}
for every $z\in (0, T]^2$. \\
\noindent  It remains to check Novikov's condition. Using (\ref{eq:5}), (\ref{case:21}) and (\ref{novikov:1}), we have
\begin{eqnarray*}
\psi_{s, t}&=& s^{-a}t^{-b} I^{a, b}_{0^{+}}(s^{a}t^{b} u_{s, t}) \\
&=&s^{-a}t^{-b} I^{a, b}_{0^{+}}(s^{a}t^{b}b_{s, t}) \\
&+& s^{-a}t^{-b}I^{a, b}_{0^{+}} \left( s^{a-a^{\prime}}t^{b-b^{}\prime}\sum_{n=1}^{\infty} (-1)^{n}  \left(D^{a^{\prime}, b^{\prime}}_{0^{+}}  s^{a^{\prime}-a}t^{b^{\prime}-b} I^{a, b}_{0^{+}}(s^{a-a^{\prime}}t^{b-b^{\prime}} \cdot)\right)^{n}(s^{a^{\prime}}t^{b^{\prime}}b_{s, t})\right).
\end{eqnarray*}
This implies
\begin{eqnarray} \label{e:N} |\psi_{s, t}| \leq |s^{-a}t^{-b}I^{a, b}_{0^{+}}(s^{a}t^{b}b_{s, t})| +\mathcal{J}_{s, t},
\end{eqnarray}
where
\[ \mathcal{J}_{s, t}:= \left| s^{-a}t^{-b}I^{a, b}_{0^{+}} \left( s^{a-a^{\prime}}t^{b-b^{\prime}}\sum_{n=1}^{\infty} (-1)^{n}  \left(D^{a^{\prime}, b^{\prime}}_{0^{+}}  s^{a^{\prime}-a}t^{b^{\prime}-b} I^{a, b}_{0^{+}}(s^{a-a^{\prime}}t^{b-b^{\prime}} \cdot)\right)^{n}(s^{a^{\prime}}t^{b^{\prime}}b_{s, t})\right)\right|.
\]
Let us first estimate the term $|s^{-a}t^{-b}I^{a, b}_{0^{+}}(s^{a}t^{b}b_{s, t})|$. Using the assumption  \eqref{eq:LG}, we have
\begin{eqnarray}\label{e:first}
|s^{-a}t^{-b}I^{a, b}_{0^{+}}(s^{a}t^{b}b_{s, t})| &\leq & || b||_{\infty}s^{-a} t^{-b} I^{a}_{0^{+}}(s^{a}) I^{ b}_{0^{+}}(t^{b}) \nonumber\\
&= &  \dfrac{|| b||_{\infty}}{4}\dfrac{\Gamma(a)\Gamma(b)}{\Gamma(2a)\Gamma(2b)}s^{a}t^{b}.
\end{eqnarray}
Next we estimate the second term $\mathcal{J}_{s, t}$. By the linearity of the fractional integral and the triangle inequality we obtain
\begin{eqnarray}\label{eq:0}
\mathcal{J}_{s, t} & \leq & \sum_{n=1}^{\infty} \left| s^{-a}t^{-b}I^{a, b}_{0^{+}}s^{a-a^{\prime}}t^{b-b^{\prime}} \left(D^{a^{\prime}, b^{\prime}}_{0^{+}}  s^{a^{\prime}-a}t^{b^{\prime}-b} I^{a, b}_{0^{+}}(s^{a-a^{\prime}}t^{b-b^{\prime}} \cdot)\right)^{n}(s^{a^{\prime}}t^{b^{\prime}}b_{s, t})\right|.
\end{eqnarray}

Define $f_n(s, t)$, $G^n(s, t)$ and $\mathcal{J}^{n}_{s, t}$ respectively by

$$f^{n} (s,t):=  s^{a - a'} t^{b - b'}   \big( D_{0+}^{a', b'} \, s^{a' - a} t^{b' - b} \,
I_{0+}^{a,b} ( s^{a - a'} t^{b - b'} \cdot ) \big)^{n}\big( s^{a'} t^{b'} b_{s,t} \big),$$
$$G^n(s, t):=I_{0+}^{a,b} (f_n(s, t)\,\,\, \,\,\, \, \mbox{and}\,\, \,\,\, \, \mathcal{J}^{n}_{s, t}:= s^{- a} t^{- b} | G^n(s, t)|.$$
Observe that 
 \begin{eqnarray}
   \label{Recurence:G_n}
   G^{n+1}(s, t)=  I_{0+}^{a,b} ( s^{a - a'} t^{b - b'}  \big( D_{0+}^{a', b'} \, s^{a' - a} t^{b' - b}  G^{n}(s, t) \big).
\end{eqnarray}
So (\ref{eq:0}) becomes
\begin{eqnarray}\label{eq:0}
\mathcal{J}_{s, t} \leq \sum_{n=1}^{\infty} \mathcal{J}^{n}_{s, t}.
\end{eqnarray}
In order to estimate  $\mathcal{J}^{n}_{s, t}$ we will make use of  Proposition \ref{prop:f_n} below. For simplicity, we start by fixing some notations. We define, for any $n \in \mathbb{N}$,  $$ \alpha_n:= (n+1)a -na' +1,\quad  \beta_n= (n+1)b -nb' +1, $$
$$\kappa_n:= \dfrac{\Gamma(\alpha_n)}{\Gamma(\alpha_n +a )}\dfrac{\Gamma(\beta_n)}{\Gamma(\beta_n +b )}, \qquad \tilde{\kappa}_n:= \dfrac{\Gamma(\alpha_n)}{\Gamma(\alpha_n +a-a' )}\dfrac{\Gamma(\beta_n)}{\Gamma(\beta_n +b-b' )} ,$$
\[d^{(n)}_1:=\int_{0}^{1}\dfrac{1-u^{a'-a}}{(1-u)^{a'+1}}u^{\alpha_n+a-1}du \quad\mbox{and}\quad d^{(n)}_2:=\int_{0}^{1}\dfrac{1-u^{b'-b}}{(1-u)^{b'+1}}u^{\beta_n+b-1}du . \]
\begin{prop}\label{prop:f_n} For all $s, t \in [0, T]$, and  $n\in \mathbb{N}$, we have
\begin{align} \label{est: f_n}
|f_n(s, t)| \leq C_n s^{(n+1)a-na'}t^{(n+1)b-nb'},
\end{align}
where the sequence $(C_n)$ is given by
 \begin{eqnarray}
\left\{\begin{array}{l}
C_0 =|| b||_{\infty},\notag \\
\\
C_{n+1}:=  \left\{ (1+c_3d^{(n)}_1d^{(n)}_2)\tilde{\kappa}_n + (c_1d^{(n)}_1 +c_2d^{(n)}_2) \kappa_n +\dfrac{c_3(d_5d^{(n)}_1 +d_6d^{(n)}_2)}{ab\Gamma(a)\Gamma(b) }\right\}C_n=: r_n C_n,
\end{array}\right.
\end{eqnarray}
with \[c_1:=\dfrac{a'}{\Gamma(1-a')\Gamma(1-b')}, c_2:=\dfrac{b'}{\Gamma(1-a')\Gamma(1-b')},  \, c_3:=\dfrac{a'b'}{\Gamma(1-a')\Gamma(1-b')} ,\]
\[d_5:= \int_{0}^{1} \dfrac{|1-v|^{b}+|t^{b}-v^{b}|}{(1-v)^{b'+1}}dv,\,\, \mbox{and}\, \,\,d_6:= \int_{0}^{1} \dfrac{|1-u|^{a}+|1-u^{a}|}{(1-u)^{a'+1}}du. \]
\end{prop}
With  Proposition \ref{prop:f_n} in mind, the estimation of $\mathcal{J}^{n}_{s, t} $ is given in the following proposition.
\begin{prop} \label{proposiiton:J_n} For all $s, t \in [0, T]$, and  $n\in \mathbb{N}$, we have
\begin{eqnarray} \label{estim:J_n} 
 \mathcal{J}^n_{s, t}  & \leq & C_n^{\star} s^{(n+1)a-na'}t^{(n+1)b-nb'},
 \end{eqnarray}
 where the sequence $(C^{\star}_n)$ is given by
 \begin{eqnarray}
\left\{\begin{array}{l}
C^{\star}_0 =c^{\star}_0,\notag \\
\\
C^{\star}_{n+1}:=  \left\{\kappa'_n + \tilde{\kappa'}_n(c_1d^{(n)}_1+ c_2d^{(n)}_2 +c_3d^{(n)}_1d^{(n)}_2)\right\}C^{\star}_n +\dfrac{c_3(d_5 d^{(n)}_1 +d_6d^{(n)}_2)}{ab\Gamma(a)\Gamma(b)}C_n\\
\qquad \,\,\,\,=:   m_n C^{\star}_n +l_nC_n,
\end{array}\right.
\end{eqnarray}
with  \[\kappa'_n:= \dfrac{\Gamma(\alpha_n + a)}{\Gamma(\alpha_n +2a-a' )}\dfrac{\Gamma(\beta_n + b)}{\Gamma(\beta_n +2b-b' )}, \qquad \tilde{\kappa}'_n:= \dfrac{\Gamma(\alpha_n + a-a')}{\Gamma(\alpha_n +2a-a' )}\dfrac{\Gamma(\beta_n +b -b')}{\Gamma(\beta_n +2b-b' )}.\]
\end{prop}
\begin{Proof} 
We employ an induction argument on $n$. For $n=0$, we have
\[
 \mathcal{J}^0_{s, t}:= s^{-a}t^{-b} |I^{a, b}_{0+}( s^{a}t^{b}b_{s, t} )| \leq  \dfrac{||b||_{\infty}}{4}\dfrac{\Gamma(a)\Gamma(b)}{\Gamma(2a)(2b)}s^{a}t^b=:c^{\star}_0s^{a}t^{b}.
\]
Let $n \in \mathbb{N}$.  Assuming (\ref{estim:J_n}) to hold for $n$, we will prove it for $n+1$.\\
\noindent Recall that
\begin{eqnarray}\label{e:J_{n+1}}
 \mathcal{J}^{n+1}_{s, t}:= s^{-a}t^{-b} | G^{n+1}(s, t)|= s^{-a}t^{-b} \left| I^{a, b}_{0+}\left( s^{a-a'}t^{b-b'}  D^{a^{\prime}, b^{\prime}}_{0^{+}} ( s^{a^{\prime}-a}t^{b^{\prime}-b}G^{n}(s, t)) \right)\right|,
\end{eqnarray}
and observe that for a function $g_{s, t}$, in view of Weil representation (\ref{fweil-rep}), we have 
\begin{eqnarray}
   \label{weil-rep}
    &&D^{a^{\prime}, b^{\prime}}_{0^{+}}( s^{a-a^{\prime}}t^{b-b^{\prime}}g_{s, t})=  s^{a-a^{\prime}}t^{b-b^{\prime}}D^{a^{\prime}, b^{\prime}}_{0^{+}}g_{s, t}  \nonumber\\   \nonumber
     & &+  \dfrac{1}{\Gamma(1- a^{\prime})\Gamma(1-b^{\prime})} \left( \dfrac{a^{\prime}}{t^b} \int_{0}^{s} \dfrac{s^{a^{\prime}-a} - u^{a^{\prime}-a}}{(s-u)^{a^{\prime} +1 }} g_{u, t}du + \dfrac{b^{\prime}}{s^{a}}\int_{0}^{t} \dfrac{t^{b^{\prime}-b} - v^{b^{\prime}-b}}{(t-v)^{b^{\prime} +1 }} g_{s, v}dv \right.
   \\
   & & + \left. a^{\prime}b^{\prime} \int_{0}^{s}\int_{0}^{t} \dfrac{\Delta_{(s, t)}(g_{u, v}u^{a^{\prime}-a}v^{b^{\prime}-b})-s^{a^{\prime}-a}t^{b^{\prime}-b}\Delta_{(s, t)}(g_{u, v})}{(s-u)^{a^{\prime} +1} (t-v)^{b^{\prime} +1}} du dv\right)\mathbbm{1}_{(0, T)^{2}}(s, t), \nonumber\\
\end{eqnarray}
where  $\Delta_{(s, t)}h(u, v):=h(s, t)-h(s, v)-h(u, t)+h(u, v)$ for any function $h$. The double integral in (\ref{weil-rep}) is also given by 
\begin{align*}
&\iint_{[0,s] \times [0,t]} (s - u)^{-a'-1} (t - v)^{-b'-1} \bigg[
\left( s^{a' - a} - u^{a' - a} \right) t^{b' - b} \, g(u, t) (u,t) \notag  \\  
&\quad + s^{a' - a} \left( t^{b' - b} - v^{b' - b} \right) g(s,v)
+ \left( u^{a' - a} v^{b' - b} - s^{a' - a} t^{b' - b} \right) g(u, v)
\bigg] \, du \, dv.
\end{align*}
Notice that the function inside this double integral can be  rewritten as
\begin{align}\label{e:doub-integ}
& \bigg[
\left( s^{a' - a} - u^{a' - a} \right) t^{b' - b} \, g (u,t) \notag + s^{a' - a} \left( t^{b' - b} - v^{b' - b} \right) g(s,v)
+ \left( u^{a' - a} v^{b' - b} - s^{a' - a} t^{b' - b} \right) g (u, v)
\bigg] \notag  \\  
& = \left( s^{a' - a}  - u^{a' - a}  \right)  \left( t^{b' - b}  - v^{b' - b}  \right) g (u, v)\notag  \\  
&\quad + \left( s^{a' - a}  - u^{a' - a}  \right)  t^{b' - b} \left(  g (u, t)  - g (u, v)   \right) \\ 
&\quad + \left( t^{b' - b}  - v^{b' - b}  \right)  s^{a' - a} \left(  g (s, v)  - g (u, v)   \right)\notag .
\end{align}
From this and by applying (\ref{weil-rep}) to the Riemann-Liouville derivative  $D^{a^{\prime}, b^{\prime}}_{0^{+}} ( s^{a^{\prime}-a}t^{b^{\prime}-b}G^{n}(s, t))$ in (\ref{e:J_{n+1}}), we obtain the following inequality
\begin{align} \label{e:J decomp}
\mathcal{J}_{s,t}^{n+1} \leq\ 
& s^{-a} t^{-b} I_{0+}^{a-a',\, b-b'} \left( s^{a} t^{b} \mathcal{J}_{s,t}^{n} \right) \notag \\
& + \dfrac{a'}{\Gamma(1-a')\Gamma(1-b')}
\left| s^{-a} t^{-b} I_{0+}^{a,b} \left(
s^{a - a'} t^{b - b'} \int_0^s \frac{s^{a' - a} - u^{a' - a}}{(s - u)^{a'+1}} u^{a} \mathcal{J}_{u,t}^{n} \, du \right) \right| \notag \\
& + \dfrac{b'}{\Gamma(1-a')\Gamma(1-b')}
\left| s^{-a} t^{-b} I_{0+}^{a,b} \left(
s^{a - a'} t^{b - b'} \int_0^t \frac{t^{b' - b} - v^{b' - b}}{(t - v)^{b'+1}} v^{b} \mathcal{J}_{s,v}^{n} \, dv \right) \right| \notag \\
& +\dfrac{a'b'}{\Gamma(1-a')\Gamma(1-b')}
\left| s^{-a} t^{-b} I_{0+}^{a,b} \left(
s^{a - a'} t^{b - b'} \mathcal{H}_n(s, t) \right) \right| \notag \\
& =: K_{s,t}^{1,n+1} + K_{s,t}^{2,n+1} + K_{s,t}^{3,n+1} + K_{s,t}^{4,n+1} ,
\end{align}
where inside the last term $K_{s,t}^{4,n+1}$, the term  $\mathcal{H}_n(s, t)$ is defined by
\begin{align}\label{H^n}
& \iint_{[0,s] \times [0,t]} du dv(s - u)^{-a'-1} (t - v)^{-b'-1} \notag \\
&\hspace{-.3cm}\times\bigg[
\left( s^{a' - a} - u^{a' - a} \right) t^{b' - b} \, G^{n} (u,t) \notag + s^{a' - a} \left( t^{b' - b} - v^{b' - b} \right) G^{n} (s,v)
+ \left( u^{a' - a} v^{b' - b} - s^{a' - a} t^{b' - b} \right) G^{n} (u, v)
\bigg]\notag \\
& =  \iint_{[0,s] \times [0,t]} (s - u)^{-a'-1} (t - v)^{-b'-1} \left( s^{a' - a}  - u^{a' - a}  \right)  \left( t^{b' - b}  - v^{b' - b}  \right) G^{n} (u, v) \, du \, dv \notag \\
&\quad +  \iint_{[0,s] \times [0,t]} (s - u)^{-a'-1} (t - v)^{-b'-1} \left( s^{a' - a}  - u^{a' - a}  \right)  t^{b' - b} \left(  G^{n} (u, t)  - G^{n} (u, v)   \right)  \, du \, dv\notag \\
&\quad +  \iint_{[0,s] \times [0,t]}  (s - u)^{-a'-1} (t - v)^{-b'-1}\left( t^{b' - b}  - v^{b' - b}  \right)  s^{a' - a} \left(  G^{n} (s, v)  - G^{n} (u, v)   \right) \, du \, dv\notag \\
& =: L_{s,t}^{1,n+1} + L_{s,t}^{2,n+1} + L_{s,t}^{3,n+1} .
\end{align}
In view of (\ref{e:J decomp}), estimating $\mathcal{J}_{s,t}^{n+1}$ reduces to estimating each  term $K_{s,t}^{i,n+1}$, $i=1, 2, 3$. The next lemma gives the necessary estimates.
\begin{lem}\label{est: K_n} We have  the following estimates
\begin{itemize}
\item[i)-]  \[ K_{s,t}^{1,n+1} \leq C^{\star}_n  \kappa'_n s^{(n+2)a-(n+1)a'} t^{(n+2)b-(n+1)b'} , \]
\item[ii)-]  \[ K_{s,t}^{2,n+1} \leq  c_1 d^{(n)}_1\tilde{\kappa}'_nC^{\star}_n  s^{(n+2)a-(n+1)a'} t^{(n+2)b-(n+1)b'} ,\]
\item[iii)-]  \[ K_{s,t}^{3,n+1} \leq  c_2 d^{(n)}_2\tilde{\kappa}'_nC^{\star}_n  s^{(n+2)a-(n+1)a'} t^{(n+2)b-(n+1)b'} , \]
\item[iv)-]  \[ K_{s,t}^{3,n+1} \leq  c_3 \left( d^{(n)}_1d^{(n)}_2 \tilde{\kappa}'_nC^{\star}_n  + \dfrac{d_5d^{(n)}_1 +d_6d^{(n)}_2 }{ab\Gamma(a)\Gamma(b)}\tilde{\kappa}'_n C_n \right)  s^{(n+2)a-(n+1)a'} t^{(n+2)b-(n+1)b'} . \]
\end{itemize}
\end{lem} 
\noindent \textbf{Proof  of Lemma \ref{est: K_n}}: $i)-$ For $K_{s,t}^{1,n+1}$, by using the induction hypothesis, we have
\begin{align*} 
&K_{s,t}^{1,n+1}\leq\ 
C^{\star}_ns^{-a} t^{-b} I_{0+}^{a-a',\, b-b'} \left( s^{(n+2)a-na'} t^{(n+2)b-b'} \right) \notag \\
&= \dfrac{C^{\star}_n\Gamma((n+2)a-na'+1)\Gamma((n+2)b-nb'+1)}{\Gamma((n+3)a-(n+1)a'+1)\Gamma((n+3)b-(n+1)b+1')}s^{(n+2)a-(n+1)a'} t^{(n+2)b-(n+1)b'} \notag \\
&= C^{\star}_n\dfrac{\Gamma(\alpha_n +a)\Gamma(\beta_n +b)}{\Gamma(\alpha_n +2a-a')\Gamma(\beta_n +2b-b')}s^{(n+2)a-(n+1)a'} t^{(n+2)b-(n+1)b'} \notag \\
&=C^{\star}_n \kappa'_{n} s^{(n+2)a-(n+1)a'} t^{(n+2)b-(n+1)b'}.
\end{align*}
$ii)-$ For $K_{s,t}^{2, n+1}$, by using the induction assumption, we have
 \begin{align*} 
&K_{s,t}^{2,n+1} = c_1  \left| s^{-a}t^{-b} I_{0+}^{a,\, b} \left(  s^{a-a'}t^{b-b'}    \int_{0}^{s}   \dfrac{s^{a'-a}-u^{a'-a}}{(s-u)^{a'+1}} u^{a}\mathcal{J}_{u,t}^{n} du \right) \right| \\
&\leq\ 
 c_1 C^{\star}_ns^{-a}t^{-b}\times 
I_{0+}^{a,\, b} \left(s^{a-a'} t^{b-b'}\int_{0}^{s} \dfrac{|u^{a'-a}-s^{a'-a}|}{(s-u)^{a'+1}} u^{a}  u^{(n+1)a-na'} t^{(n+1)b-nb'} \right) \notag \\
&=c_1  d_1^{(n)}C^{\star}_n \dfrac{\Gamma(\alpha_n +a-a')\Gamma(\beta_n +b-b')}{\Gamma(\alpha_n +2a-a')\Gamma(\beta_n +2b-b')}s^{(n+2)a-(n+1)a'} t^{(n+2)b-(n+1)b'}\notag \\
&=c_1 d_1^{(n)}C^{\star}_n \tilde{\kappa}'_{n} s^{(n+2)a-(n+1)a'} t^{(n+2)b-(n+1)b'}.
\end{align*}
$iii)-$ Similarly, for $K_{s,t}^{3, n+1}$ we have
 \begin{align*} 
&K_{s,t}^{3,n+1} \leq\ 
 c_2 C^{\star}_ns^{-a}t^{-b}\times 
I_{0+}^{a,\, b} \left(s^{a-a'} t^{b-b'}\int_{0}^{s} \dfrac{|t^{b'-b}-v^{b'-b}|}{(t-v)^{b'+1}} v^{b}  v^{(n+1)b-nb'} s^{(n+1)a-na'} \right) \notag \\
&=c_2  d_2^{(n)}C^{\star}_n \dfrac{\Gamma(\alpha_n +a-a')\Gamma(\beta_n +b-b')}{\Gamma(\alpha_n +2a-a')\Gamma(\beta_n +2b-b')}s^{(n+2)a-(n+1)a'} t^{(n+2)b-(n+1)b'}\notag \\
&=c_2 d_2^{(n)}C^{\star}_n \tilde{\kappa}'_{n} s^{(n+2)a-(n+1)a'} t^{(n+2)b-(n+1)b'}.
\end{align*}
$iv)-$ To estimate $K_{s,t}^{3,n+1}$ we need the following claim whose proof is postponed to Section \ref{Sect:Tech}.
\begin{claim} The following estimates hold
\begin{itemize}
\item[a)-] $$s^{-a}t^{-b} I_{0+}^{a,\, b}(s^{a-a'}t^{b-b'}|L^{1, n+1}_{s, t}|) \leq d_1^{(n)}d_2^{(n)}\tilde{\kappa}'_{n}C^{\star}_n s^{(n+2)a-(n+1)a'} t^{(n+2)b-(n+1)b'} ,$$
\item[b)-] $$s^{-a}t^{-b} I_{0+}^{a,\, b}(s^{a-a'}t^{b-b'}|L^{2, n+1}_{s, t}|) \leq \dfrac{d_5d_1^{(n)}C_n}{ab\Gamma(a)\Gamma(b)}\tilde{\kappa}'_{n} s^{(n+2)a-(n+1)a'} t^{(n+2)b-(n+1)b'} ,$$
\item[c)-] $$s^{-a}t^{-b} I_{0+}^{a,\, b}(s^{a-a'}t^{b-b'}|L^{3, n+1}_{s, t}|) \leq \dfrac{d_6d_2^{(n)}C_n}{ab\Gamma(a)\Gamma(b)}\tilde{\kappa}'_{n} s^{(n+2)a-(n+1)a'} t^{(n+2)b-(n+1)b'} .$$
\end{itemize}
\end{claim}
With this claim, we have 
 \begin{align*} 
&K_{s,t}^{4,n+1}\leq\ 
 c_3\left| s^{-a} t^{-b} I_{0+}^{a,b} \left(
s^{a - a'} t^{b - b'} \mathcal{H}_n(s, t) \right) \right|  \notag \\
&=c_3 \sum_{i=1}^{3} s^{-a}t^{-b} I_{0+}^{a,b} (s^{a-a'}t^{b-b'}|L^{i, n+1}_{s, t}|) \notag \\
& \leq \left(c_3  d_1^{(n)}d_2^{(n)} \tilde{\kappa}'_{n}C^{\star}_n  + \dfrac{c_3(d_5d_1^{(n)}+ d_6d_2^{(n)})\tilde{\kappa}'_{n} }{ab\Gamma(a)\Gamma(b)}C_n\right) s^{(n+2)a-(n+1)a'} t^{(n+2)b-(n+1)b'}.
\end{align*}
This finishes the proof of Lemma \ref{est: K_n}. 

Next, we give a proof of (\ref{estim:J_n}) for $n+1$. Using the decomposition (\ref{e:J decomp}) and in view of Lemma  \ref{est: K_n}, we deduce
 \begin{align*} 
&\mathcal{J}^{n+1}_{s, t}\leq\ 
 \sum_{i=1}^{3} K_{s,t}^{i,n+1} \notag\\
 &\leq  \left\{C^{\star}_n\left( \kappa'_{n} + c_1d_1^{(n)} \tilde{\kappa}'_{n}+c_2d_2^{(n)} \tilde{\kappa}'_{n}+  c_3 d_1^{(n)}d_2^{(n)} \tilde{\kappa}'_{n}     \right)\right.\notag\\
 &\left.+ c_3C_n\tilde{\kappa}'_{n}\dfrac{d_5d_1^{(n)}+ d_6d_2^{(n)}}{ab\Gamma(a)\Gamma(b)}\right\} s^{(n+2)a-(n+1)a'} t^{(n+2)b-(n+1)b'}\notag\\
 &=:C^{\star}_{n+1}s^{(n+2)a-(n+1)a'} t^{(n+2)b-(n+1)b'}.
\end{align*}
This is (\ref{estim:J_n})  for $n+1$. This concludes the inductive step and hence the proof of Proposition \ref{proposiiton:J_n}.

We  are now in position to complete the proof of Novikov's condition. From (\ref{e:expressionC*}) and (\ref{est: 2seqC'}), we deduce that there exists $A>1$ such that 
\begin{align} \label{e:second} 
& \sup_{s, t \leq T} |\mathcal{J}_{s, t}| \leq A(1+ c_0^{\star}) +  Ac_0^{\star} \leq 2A(1+ c_0^{\star}) .
\end{align}
By plugging (\ref{e:first}) and (\ref{e:second}) into (\ref{e:N}) we get
\begin{eqnarray*} |\psi_{s, t}| \leq 2A(1+ c_0^{\star}) + c(1+ ||B^{\alpha, \beta} +B^{\alpha^{\prime}, \beta^{\prime}} ||_{\infty})
\end{eqnarray*}
for some constant $c>0$. Now, since $c_0^{\star}:=\dfrac{\Gamma(a)\Gamma(b)}{4\Gamma(2a)\Gamma(2b)}||b||_{\infty}$, we deduce by using the linear growth condition of $b$ that there exists $c>0$ such that
\begin{eqnarray} |\psi_{s, t}| \leq  c(1+ ||B^{\alpha, \beta} +B^{\alpha^{\prime}, \beta^{\prime}} ||_{\infty}).
\end{eqnarray}
Therefore,  Novikov's condition holds, which completes the proof of the existence of $u$ and $v$.
\end{Proof}
\begin{center}
\section{Uniqueness in law and Pathwise uniqueness }\label{four}
\end{center}
The results in this section are obtained by adapting the proofs of  \cite{NS22} and  \cite{NO02}  to the double parameter setting where the driving noise is the sum of two fractional Brownian sheets. We have the following statement.
\begin{theorem}\label{thm:uniq in law} 
Let $(\alpha, \beta)$ and $(\alpha', \beta')$ in $(0, 1/2]^{2}$.  We assume that  $(\alpha, \beta)\prec (\alpha^{\prime}, \beta^{\prime})  \preceq (\frac{1}{2}, \frac{1}{2})$  or $(\alpha^{\prime}, \beta^{\prime}) \prec(\alpha, \beta)  \preceq (\frac{1}{2}, \frac{1}{2})$ and that the Borel function $b$ satisfies the linear growth condition \eqref{eq:LG}. Then two weak solutions of Eq. (\ref{eq:1}) have the same distribution.
\end{theorem} 
\begin{Proof}  We only give the proof in the case $(\alpha, \beta)\prec (\alpha^{\prime}, \beta^{\prime}) \preceq (\dfrac{1}{2}, \dfrac{1}{2})$. The idea of proof in all remaining cases is essentially the same.
Let $( X, B^{\alpha, \beta}, B^{\alpha^{\prime}, \beta^{\prime}})$ be a weak solution of Eq. (\ref{eq:1}) defined on a filtered probability space $(\Omega, \mathcal{F}, (\mathcal{F}_z)_{z\in [0, T]^2}, \mathbb{P})$ with underlying Brownian sheet $(W_z)_{z\in [0, T]^2}$. With the notations of Subsection \ref{cons:u and v}, we define the two processes $u$, $v$ by
\begin{equation}\label{e:u unicity}
s^{-a^{\prime}}t^{-b^{\prime}} \left[I +D^{a^{\prime}, b^{\prime}}_{0^{+}} s^{a^{\prime}-a}t^{b^{\prime}-b} I^{a, b}_{0^{+}}(s^{a-a^{\prime}}t^{b-b^{\prime}} \cdot) \right]^{-1}(s^{a^{\prime}}t^{b^{\prime}}b(s, t, X_{s, t}))
\end{equation}
and 
  \begin{equation}\label{e:v unicity}
v_{s, t} = s^{-a^{\prime}}t^{-b^{\prime}} D^{a^{\prime}, b^{\prime}}_{0^{+}} \left(s^{a^{\prime}-a}t^{b^{\prime}-b} I^{a, b}_{0^{+}} (s^{a}t^{b} u_{s, t})\right),
 \end{equation}
for every  $z=(s, t) \in (0, T]^2$. Notice  that  we have  $u_{s, t}+ v_{s, t}= b(s, t, X_{s, t})$ and 
\begin{equation}
(\mathcal{K}^{\alpha, \beta})^{-1} \left( \int_{[0, \cdot]} u_{\zeta} d \zeta \right)(z)= (\mathcal{K}^{\alpha^{\prime}, \beta^{\prime}})^{-1}  \left( \int_{[0, \cdot]} v_{\zeta} d \zeta \right)(z)=:\psi_z,
\end{equation}
for all $z=(s, t) \in (0, T]^2$. Let $\mathbb{Q}$ be the probability measure  on $\mathcal{F}_{(T, T)}$ defined by 
\[
\left.\dfrac{d\mathbb{Q}}{d\mathbb{P}}\right|_{\mathcal{F}_{(T, T)}}=\exp \left( - \int_{[0, T]^2}\psi_zdW_z -\dfrac{1}{2}  \int_{[0, T]^2} \psi_z^2 dz \right).\]
The processes $u$ and $v$ are both $(\mathcal{F}_{z})$-adapted and by construction satisfy the conditions of Theorem \ref{thm2}. In fact, arguing as in \cite{ENO03} we obtain by Gronwall's lemma for integrals in the plane that
\begin{eqnarray*}
|| X||_{\infty}\leq (|x| + || B^{\alpha, \beta}||_{\infty} +|| B^{\alpha, \beta}||_{\infty} +cT)e^{cT^2} .
\end{eqnarray*}
Hence, 
\begin{eqnarray*}
|\psi_{s, t}| \leq c(1+ || X||_{\infty})\leq c (1 + || B^{\alpha, \beta}||_{\infty} +|| B^{\alpha, \beta}||_{\infty} )
\end{eqnarray*}
and then Novikov's condition is fulfilled. Therefore by Theorem \ref{thm2}, the process $(\tilde{W}_z)$ defined by 
\[\tilde{W}_z:= W_z +\int_{[0, z]} \psi_{\xi}\, d \xi, \qquad \mbox{for all}\, \,  z\in [0, T]^2\]
is a standard Brownian sheet under the probability measure $\mathbb{Q}$. Next, note that we have, for all $z\in [0, T]^2$,
\begin{align}
 X_z&= x + B^{\alpha, \beta}_z + B^{\alpha', \beta'}_z + \int_{[0, z]} b(\zeta, X_{\zeta})\, d \zeta\nonumber\\
&= x+ \int_{[0, z]} K^{\alpha, \beta}(z, \zeta)d\tilde{W}_{\zeta} + \int_{[0, z]} K^{\alpha', \beta'}(z, \zeta)d\tilde{W}_{\zeta}\notag\\
&=: x+  \tilde{B}^{\alpha, \beta}_z +  \tilde{B}^{\alpha', \beta'}_z.
\end{align}
Thus $X-x$ is the sum of two fractional Brownian sheets with Hurst parameters respectively $(\alpha, \beta)$ and $(\alpha', \beta')$. As a consequence, for every measurable bounded functional $\Phi: \mathcal{C(\mathbb{R})} \longrightarrow \mathbb{R}_{+}$, we get
\begin{align} \label{e:law of X}
& \mathbb{E} [ \Phi(X-x)]= \mathbb{E}_{\mathbb{Q}}\left[\Phi(\tilde{B}^{\alpha, \beta}+\tilde{B}^{\alpha', \beta'} ) \exp\left(  \int_{[0, z]}  \psi_z d \tilde{W}_z -\frac{1}{2}  \int_{[0, z]}  \psi_z^2 dz \right)\right].
\end{align}

From (\ref{e:u unicity}), (\ref{e:v unicity}) and (\ref{e:law of X}),  $u$ and $v$ can be written in terms  of  $\tilde{B}^{\alpha, \beta}$ and  $\tilde{B}^{\alpha', \beta'}$ and consequently the same holds for $\psi$. Therefore, (\ref{e:law of X}) becomes
\begin{align} 
& \mathbb{E} [ \Phi(X-x)]= \mathbb{E}\left[\Phi(B^{\alpha, \beta}+B^{\alpha', \beta'} ) \exp\left(  \int_{[0, z]}  \phi_z d W_z -\frac{1}{2}  \int_{[0, z]}  \phi^2_z dz \right)\right],
\end{align}
with \[\phi_z:= (\mathcal{K}^{\alpha, \beta})^{-1} \left( \int_{[0, \cdot]} l_{\zeta} d \zeta \right)(z)\]
and 
\begin{equation*}
l_{s, t}:=s^{-a^{\prime}}t^{-b^{\prime}} \left[I +D^{a^{\prime}, b^{\prime}}_{0^{+}} s^{a^{\prime}-a}t^{b^{\prime}-b} I^{a, b}_{0^{+}}(s^{a-a^{\prime}}t^{b-b^{\prime}} \cdot) \right]^{-1}(s^{a^{\prime}}t^{b^{\prime}}b(s, t, x+ B^{\alpha, \beta}_{s, t} + B^{\alpha', \beta'}_{s, t})).
\end{equation*}
Thus, we obtain the uniqueness in law.
\end{Proof} 
As a corollary of Theorem \ref{thm:uniq in law}, we obtain the pathwise uniqueness property for solutions of Eq. (\ref{eq:1})
\begin{corollary}\label{cor:1} Suppose the assumptions of Theorem \ref{thm:uniq in law} are satisfied. Then two weak solutions defined on the same probability space coincide almost surely. 
\end{corollary}
\begin{center}
\section{Existence of Strong solutions}\label{five}
\end{center}
In this section, we prove the existence and uniqueness of a strong solution of Eq. (\ref{eq:1}).  In view of Corollary \ref{cor:1} it suffices to show the existence of a strong solution. We first establish  a Krylov-type inequality that plays a prominent role in the proof of the existence.
\begin{prop}  \label{prop:Krylov}
Let $(\alpha, \beta)$ and $(\alpha', \beta')$ be in $(0, 1/2]^{2}$ such that $(\alpha, \beta)\prec (\alpha^{\prime}, \beta^{\prime}) \preceq (\frac{1}{2}, \frac{1}{2})$.  We suppose that the Borel function $b$  is uniformly bounded. If $X$ is a weak solution of Eq. (\ref{eq:1}) and $\rho > 1 +\max( \alpha, \beta)$, then there exists a constant $C>0$ depending only on $T$, $||b||_{\infty}$, $\rho$  and not on $X$, such that for every bounded measurable function $g: [0, T]^2 \times \mathbb{R} \longrightarrow \mathbb{R}_{+}$ we have
\begin{align*}
&\E\left[  \int_{[0, T]^2}g(\zeta, X_{\zeta}) d\zeta\right] \leq C \left(\int_{[0, T]^2} \int_{\mathbb{R}} g^{\rho}( \zeta, y) dy d \zeta \right)^{1/\rho}.
\end{align*}
\end{prop}
\begin{Proof} We write $b= u + v$ and define $\psi$ as in the proof of Theorem \ref{thm2} (see, Subsection \ref{existence of u and v}). Let
\[\ \Theta:= \exp\left(  - \int_{[0, T]^2}\psi_{s, t}dW_{s, t} -\dfrac{1}{2}  \int_{[0, T]^2} \psi_{s, t}^2 ds dt \right) \]
and consider the probability measure $\tilde{\P}$ defined by $\dfrac{d\tilde{\P}}{d\P}= \Theta$. Recall that we have $$|\psi_{s, t}| \leq c(1 + || b||_{\infty}) \leq M.$$
Then, Novikov's condition is satisfied and $\Theta$ is a true martingale. Hence, for all $\alpha >1$, 
\begin{align*}
& \E [ \Theta^{-\alpha}]= \E \left[ \exp\left(  \alpha \int_{[0, T]^2}\psi_{s, t}dW_{s, t} -\dfrac{\alpha^2 }{2}  \int_{[0, T]^2} \psi_{s, t}^2 ds dt \right) \right]\\
&\qquad \qquad \times \exp\left( \dfrac{\alpha(\alpha +1)}{2}\int_{[0, T]^2}\psi_{s, t}^{2}ds dt \right)\\
&\leq e\exp\left( {\dfrac{M^2 T^2 \alpha (\alpha +1)}{2}} \right)=:C.
\end{align*}
On the other hand, in view of H\"{o}lder inequality we have for $\varrho>1$ such that $\alpha^{-1} + \varrho^{-1}=1$
\begin{align}\label{est:int g}
& \E\left[  \int_{[0, T]^2}g(\zeta, X_{\zeta}) d\zeta\right]=\tilde{\mathbb{E}}\left[ \Theta^{-1}\int_{[0, T]^2}g(\zeta, X_{\zeta}) d\zeta  \right]\notag\\ 
& \leq \left(\tilde{\E}[\Theta^{-\alpha}]\right)^{1/\alpha} \left(\tilde{\E}\left[ \int_{[0, T]^2}g(\zeta, X_{\zeta}) d\zeta\right]^{\varrho} \right)^{1/\varrho}\notag\\
& \leq C\left( \tilde{\E}\left[ \int_{[0, T]^2}g^{\varrho}(\zeta, X_{\zeta}) d\zeta\right] \right)^{1/\varrho},
\end{align}
where we have used Jensen's inequality in the last line. 

Next, using Theorem \ref{thm2}, there exists a  standard Brownian sheet $\tilde{W}$ under $\tilde{\P}$ such that, for all $z\in [0, T]^2$
\begin{align*}
&X_z= x +  \int_{[0, z]} K^{\alpha, \beta}(z, \zeta)d\tilde{W}_{\zeta} + \int_{[0, z]} K^{\alpha', \beta'}(z, \zeta)d\tilde{W}_{\zeta}\notag\\
&= x + \int_{[0, z]} \left( K^{\alpha, \beta}(z, \zeta) + K^{\alpha', \beta'}(z, \zeta)\right)d\tilde{W}_{\zeta}.
\end{align*}
Then, under $\tilde{\P}$, $X_z-x$ is a Gaussian random variable with variance 
\begin{align}\label{e:sigma}
&\sigma^2(z);= \int_{[0, z]}   \left( K^{\alpha, \beta}(z, \zeta) + K^{\alpha', \beta'}(z, \zeta)\right)^2 d \zeta\notag\\
&= \int_{0}^{s}\int_{0}^{t} \left( K^{\alpha, \beta}((s, t), (u,v )) + K^{\alpha', \beta'}((s, t), (u,v ))\right)^2 du dv,
\end{align}
for all $z:=(s, t) \in [0, T]^2.$ Hence
\[ \tilde{\E}\left(\left( \int_{[0, T]^2}g^{\varrho}(\zeta, X_{\zeta}) d\zeta\right) \right)= \dfrac{1}{\sqrt{2\pi }} \int_{[0, T]^2}  \int_{\mathbb{R}} \dfrac{ g(z, x +y)^{\varrho}}{\sigma^2(z)}\exp\left(-\dfrac{y^2}{2\sigma^2(z)}\right) \, dy\, dz\]
\noindent Now, since $\rho > 1 +\max( \alpha, \beta)$, we choose  $\delta >1$ so that  $\rho> \delta > 1 +\max( \alpha, \beta)$. Notice that the conjugate $\gamma$ of $\delta$ satisfies $1<\gamma< 1 + \min(\frac{1}{\alpha}, \frac{1}{\beta})$. Thus from Lemma \ref{lem: strong}, we have
\begin{align}
&\int_{[0, T]^2} \sigma
^{1-\gamma}(z)dz <+\infty.
\end{align}
 By using H\"{o}lder inequality, we  get
\begin{align}\label{est: ud}
& \qquad \tilde{\E}\left(\int_{[0, T]^2}g^{\varrho}(\zeta, X_{\zeta}) d\zeta\right)\notag\\
& \leq   \dfrac{1}{\sqrt{2\pi }} \left(\int_{[0, T]^2}  \int_{\mathbb{R}}g^{\varrho \delta}(z, x +y ) dy dz\right)^{1/\delta}   \left(\int_{[0, T]^2}  \int_{\mathbb{R}} \sigma
^{-\gamma}(z) \exp\left(-\dfrac{\gamma y^2}{2\sigma^2(z)}\right) \right)^{1/\gamma} \notag\\
&= \dfrac{1}{\sqrt{2\pi }}  \left(\int_{[0, T]^2}  \int_{\mathbb{R}}g^{\varrho \delta}(z, x +y ) dy dz\right)^{1/\delta} \left(\int_{[0, T]^2} \sigma
^{1-\gamma}(z)dz \right)^{1/\gamma}\notag\\
&= C \left(\int_{[0, T]^2}  \int_{\mathbb{R}}g^{\varrho \delta}(z, x +y ) dy dz\right)^{1/\delta} 
\end{align}
Consequently, combining  (\ref{est:int g}) and (\ref{est: ud}) we get 
\[ \E\left[  \int_{[0, T]^2}g(\zeta, X_{\zeta}) d\zeta\right]\leq  C \left(\int_{[0, T]^2}  \int_{\mathbb{R}}g^{\varrho \delta}(z, x +y ) dy dz\right)^{1/\delta}. 
\]
Since $\varrho>1$ is arbitrary, we may choose $\varrho= \rho/ \delta>1$. The conclusion  of Proposition \ref{prop:Krylov} then follows.
\end{Proof} 
\begin{theorem}\label{thm:uniq in law} 
Let $(\alpha, \beta)$ and $(\alpha', \beta')$ be in $(0, 1/2]^{2}$.  We assume that  $(\alpha, \beta)\prec (\alpha^{\prime}, \beta^{\prime})  \preceq (\frac{1}{2}, \frac{1}{2})$  or $(\alpha^{\prime}, \beta^{\prime}) \prec(\alpha, \beta)  \preceq (\frac{1}{2}, \frac{1}{2})$ and that the Borel function $b$ satisfies the following two conditions:
\begin{itemize}
\item[(i)] $ \sup_{z\in [0, T]^2} \sup_{x \in \mathbb{R}} |b(z, x)| \leq M < \infty;$
\item[(ii)]  $x \longmapsto b(z, x)$ is a nondecreasing function for all $z \in [0, T]^2$.
\end{itemize}
Then,  Eq. (\ref{eq:1})  has a unique strong solution.
\end{theorem} 
\begin{Proof} 
Following the line of reasoning in Lemmas $10 $ and $12$ and Theorem $2$ of \cite{ENO03}, the desired conclusion follows essentially from a comparison theorem which applies since  $x \longmapsto b(z, x)$ is a nondecreasing function and  Proposition \ref{prop:Krylov} above.
\end{Proof} 
\begin{center}
\section{Technical results}\label{Sect:Tech}
\end{center}
\subsection{Proof of Proposition \ref{prop:f_n}:}
\noindent \textbf{Proof of Proposition \ref{prop:f_n}:} We will  use an induction argument. \\
\noindent For $n=0$, $f_0(s, t)=s^{a}t^{b} b_{s,t}$. So $|f(s, t)|\leq ||b||_{\infty} s^{a}t^{b}$. \\
\noindent Let $n\in \mathbb{N}$.  Assume (\ref{est: f_n}) and let us prove it for $n+1$. We have
\begin{align}\label{recursive formula:f_n}
f_{n+1}(s,t)&=s^{a - a'} t^{b - b'}    \big( D_{0+}^{a', b'} \, s^{a' - a} t^{b' - b} \,
I_{0+}^{a,b} ( s^{a - a'} t^{b - b'}  ) \big)\big( D_{0+}^{a', b'} \, s^{a' - a} t^{b' - b} \,
I_{0+}^{a,b} ( s^{a - a'} t^{b - b'} \cdot ) \big)^{n}\big( s^{a'} t^{b'} b_{s,t} \big) \notag\\
& = s^{a - a'} t^{b - b'}D_{0+}^{a', b'} \big(\, s^{a' - a} t^{b' - b} P_n(s, t) \big),
\end{align}
with
$$P_n(s, t) = I_{0+}^{a,b} ( f_{n} (s,t)).$$
By applying (\ref{weil-rep}) to   $g_{s, t}:=P_{n}(s, t)$ in (\ref{recursive formula:f_n}) and taking account of (\ref{e:doub-integ}), we obtain 
\[f_{n+1}(s, t)= \sum_{i=1}^{4} f^{i}_{n+1}(s, t),\]
where  $$f^{1}_{n+1}(s, t):= I_{0+}^{a-a',b-b'} f_n(s, t), \quad  f^{2}_{n+1}(s, t):= c_2 s^{a - a'} t^{b - b'} \int_{0}^{s}  \dfrac{s^{a^{\prime}-a} - u^{a^{\prime}-a}}{(s-u)^{a^{\prime} +1 }} P_n(u, t)du,  $$ $$  f^{3}_{n+1}(s, t):= c_1 s^{a - a'} t^{b - b'} \int_{0}^{t}  \dfrac{t^{b^{\prime}-b} - v^{b^{\prime}-b}}{(t-v)^{b^{\prime} +1 }} P_n(s, v)dv, $$
and 
\[ f^{4}_{n+1}:= \sum_{i=1}^{3}s^{a - a'} t^{b - b'} \tilde{L}^{i}_n(s, t),
\]
with
{\small{
\begin{eqnarray}
\left\{\begin{array}{l}
\tilde{L}^{1}_n(s, t) := c_3s^{a - a'} t^{b - b'}  \int_{0}^{s}\int_{0}^{t} (s - u)^{-a'-1} ( t - v )^{-b'-1} ( s^{a' - a}  - u^{a' - a})  ( t^{b' - b}  - v^{b' - b}  ) P_n (u, v)\, du \, dv ,\notag \\
\\
\tilde{L}^{2}_n(s, t) :=  c_3s^{a - a'} t^{b - b'}  \int_{0}^{s}\int_{0}^{t} (s - u)^{-a'-1} (t - v)^{-b'-1} ( s^{a' - a}  - u^{a' - a}  )  t^{b' - b} (  P_n  (u, t)  -P_n  (u, v)   )\, du \, dv,\notag \\
\\
\tilde{L}^{3}_n(s, t) := c_3s^{a - a'} t^{b - b'} \int_{0}^{s}\int_{0}^{t}(s - u)^{-a'-1} (t - v)^{-b'-1}  ( t^{b' - b}  - v^{b' - b}  )  s^{a' - a} (  P_n (s, v)  - P_n (u, v)   )\, du \, dv.
\end{array}\right.
\end{eqnarray}}}

We also need the following 
\begin{lem}\label{lem:f} Let $\gamma(n):= (n+1)a-na'$ and $\tilde{\gamma}(n):= (n+1)b-nb'$. Then, we have the following estimates:
$$|f^{1}_{n+1}(s, t)| \leq C_n \tilde{\kappa}_n s^{\gamma(n+1)} t^{\tilde{\gamma}(n +1)}, \,\, |f^{2}_{n+1}(s, t)| \leq C_nd^{(n)}_{1} \kappa_n s^{\gamma(n+1)} t^{\tilde{\gamma}(n +1)}, $$
\[|f^{3}_{n+1}(s, t)| \leq C_nd^{(n)}_{2} \kappa_n s^{\gamma(n+1)} t^{\tilde{\gamma}(n +1)}\]
and 
\[|f^{4}_{n+1}(s, t)| \leq c_3 \left(d^{(n)}_{1} d^{(n)}_{2}\tilde{\kappa}_n + \dfrac{d_5d^{(n)}_{1} +d_6 d^{(n)}_{2}}{ab\Gamma(a)\Gamma(b)} \right) C_n s^{\gamma(n+1)} t^{\tilde{\gamma}(n +1)} . \]
\end{lem}
\noindent \textbf{Proof of Lemma \ref{lem:f}}:  Let us estimate $f^{1}_{n+1}(s, t)$. Using the induction assumption, we have
\begin{align*}
|f^{1}_{n}(s,t)|&\leq C_n I_{0+}^{a-a',b-b'}  \big(s^{(n+1)a -na'}t^{(n+1)b -nb'} \big) \notag\\
& = C_n \dfrac{\Gamma(\alpha_n)}{\Gamma(\alpha_n +a-a')}\dfrac{\Gamma(\beta_n)}{\Gamma(\beta_n +b-b')}s^{\gamma(n+1)} t^{\tilde{\gamma}(n +1)}= C_n \kappa_n s^{\gamma(n+1)} t^{\tilde{\gamma}(n +1)}.
\end{align*}
For $f^{2}_{n+1}(s, t)$, by using the induction assumption, we get
\begin{align}\label{P_n: estimate}
P_n(u, t)|=  |I_{0+}^{a, b} f_n(u, t) | \leq C_n \dfrac{\Gamma(\alpha_n)\Gamma(\beta_n)}{\Gamma(\beta_n +a)\Gamma(\beta_n +b)} u^{(n+2)a-na'}t^{(n+2)b-nb'}.
\end{align}
Then
\begin{align*}
|f^{2}_{n}(s,t)|&\leq s^{a-a'}t^{-b'}C_n \kappa_n\left(\displaystyle{\int_{0}^{s}} \dfrac{|s^{a^{\prime}-a} - u^{a^{\prime}-a}|}{(s-u)^{a^{\prime} +1 }}u^{(n+2)a-na'} du \right)t^{(n+2)b-nb'} \notag\\
& = c_1C_n \kappa_n d^{(n)}_1  s^{a-a' +(n+1)a-na'} t^{-b' +(n+2)b-nb'}=c_1C_n \kappa_n d^{(n)}_1 s^{\gamma(n+1)} t^{\tilde{\gamma}(n +1)}.
\end{align*} 
For $f^{3}_{n}(s,t)$, similarly we have
\begin{align*}
|f^{3}_{n}(s,t)|&\leq c_2C_n \kappa_n t^{b-b'}s^{-a'}\left(\displaystyle{\int_{0}^{t}} \dfrac{|t^{b^{\prime}-b} - v^{b^{\prime}-b}|}{(t-v)^{b^{\prime} +1 }}v^{(n+2)b-nb'} dv \right)= c_2d^{(n)}_2C_n \kappa_ns^{\gamma(n+1)} t^{\tilde{\gamma}(n +1)}.
\end{align*} 
To estimate $f^{4}_{n}(s,t)$ we need the following claim whose proof is postponed until the proof of Proposition \ref{prop:f_n} is finished.
\begin{claim}\label{claim: L'} Let $\gamma(n):= (n+1)a-na'$ and $\tilde{\gamma}(n):= (n+1)b-nb'$. Then, we have:
\begin{itemize} 
\item[i)-]  \[s^{a-a'}t^{b-b'} |\tilde{L}^{1}_n(s, t)| \leq c_3C_n  d^{(n)}_1d^{(n)}_2\tilde{\kappa}_n s^{\gamma(n+1)} t^{\tilde{\gamma}(n +1)} ,
\]
\item[ii)-]  \[s^{a-a'}t^{b-b'} |\tilde{L}^{2}_n(s, t)| \leq c_3 \dfrac{d_5d^{(n)}_1}
{a\Gamma
(a)\Gamma(b)} C_n  s^{\gamma(n+1)} t^{\tilde{\gamma}(n +1)} ,
\]
\item[iii)-]  \[s^{a-a'}t^{b-b'} |\tilde{L}^{3}_n(s, t)| \leq c_3 \dfrac{d_6d^{(n)}_2}{a\Gamma
(a)\Gamma(b)} C_n  s^{\gamma(n+1)} t^{\tilde{\gamma}(n +1)}.
\]
\end{itemize}
\end{claim}
\noindent Now  using this claim, we have for  $f^{4}_{n}(s,t)$
\begin{align*}
|f^{4}_{n}(s,t)|&\leq s^{a-a'}t^{b-b'}  \sum_{i=1}^3 |\tilde{L}^{i}_n(s, t)|  \\ 
&\leq c_3C_n \left( \tilde{\kappa}_n  d^{(n)}_1d^{(n)}_2 +\dfrac{d_5d^{(n)}_1 + d_6d{(n)}_2}
{a\Gamma
(a)\Gamma(b)} \right)  s^{\gamma(n+1)} t^{\tilde{\gamma}(n +1)}.
\end{align*} 
This finishes the proof of Lemma \ref{lem:f}.

Therefore, in view of Lemma  \ref{lem:f}, we obtain 
\begin{align*}
|f_{n+1}(s,t)|&\leq \sum_{i=1}^{4} |f^{i}_{n+1}(s, t)\\\
&\leq C_n \left( \tilde{\kappa}_n+ c_3\tilde{\kappa}_n  d^{(n)}_1d^{(n)}_2 + \kappa_n(c_1d^{(n)}_1+ c_2d^{(n)}_2) + c_3 \dfrac{d_5d^{(n)}_1 + d_6d_2^{(n)}}
{a\Gamma
(a)\Gamma(b)} \right)  s^{\gamma(n+1)} t^{\tilde{\gamma}(n +1)}\\
&=C_{n+1}s^{\gamma(n+1)} t^{\tilde{\gamma}(n +1)}.
\end{align*} 
This finishes the proof of Proposition \ref{prop:f_n}.

\noindent \textbf{Proof of the Claim \ref{claim: L'}:} The proof of this claim follows the same line of reasoning as that in  Lemma \ref{lem:f}. Indeed, we use the induction estimate  for $f_n$ together with  Proposition~\ref{Prop:A1}  from the appendix. For  the  sake of completeness, we give the proof here. \\
\noindent For $i)$, we first recall that \[\tilde{L}^{1}_n(s, t) := c_3s^{a - a'} t^{b - b'}  \int_{0}^{s}\int_{0}^{t} \dfrac{( s^{a' - a}  - u^{a' - a})  ( t^{b' - b}  - v^{b' - b}  ) }{(s - u)^{a'+1} ( t - v )^{b'+1}}P_n (u, v)\, du \, dv .\]
Then from the inequality 
\begin{align*}
|P_n(u, v)|=  |I_{0+}^{a, b} f_n(u, t) | \leq C_n \kappa_n u^{(n+2)a-na'}v^{(n+2)b-nb'}
\end{align*}
we obtain 
\begin{align*}
&|\tilde{L}^{1}_n(s, t)| \leq c_3 C_n \kappa_n \left(\int_{0}^{s} \dfrac{u^{a'-a}-s^{a'-a}}{(s-u)^{a'+1}}u^{(n+2)a-na'} du\right)\times \left(\int_{0}^{t} \dfrac{v^{b'-b}-t^{b'-b}}{(t-v)^{b'+1}} v^{(n+2)b-nb'}dv\right)\\
&\qquad =c_3 d^{(n)}_1d^{(n)}_2 C_n \kappa_n  s^{(n+1)a-na'}t^{(n+1)b-nb'}.
\end{align*}
Thus
\[s^{a-a'}t^{b-b'}|\tilde{L}^{1}_n(s, t)| \leq c_3 d^{(n)}_1d^{(n)}_2C_n \kappa_n s^{\gamma(n+1)}t^{\tilde{\gamma}(n +1)}.\] 
\noindent For $ii)$, using (\ref{est:RF1}) from Corollary \ref{est:RF-d}, we have 
\begin{align*}
&|P_n(u, t)-P_n(u, v)| = |I_{0+}^{a, b} f_n(u, t)-I_{0+}^{a, b} f_n(u, v) |\\
&\qquad \quad \quad \leq C_n \dfrac{u^{(n+2)a-na'}t^{(n+2)b-nb'}}{ab\Gamma(a)\Gamma(b)}u^{a}[|t-v|^b+|t^b-v^b|].
\end{align*}
Then
\begin{align*}
& |\tilde{L}^{2}_n(s, t)| \leq c_3 \int_{0}^{s}\int_{0}^{t} \dfrac{|s^{a'-a}-u^{a'-a}|t^{b'-b}|P_n(u, t)-P_n(u, v)|}{(s-u)^{a'+1}+(t-v)^{b'+1}} \, du\,  dv\\
&\leq \dfrac{c_3C_n}{ab\Gamma(a)\Gamma(b) }t^{(n+1)b-nb')}t^{b'-b}\left(    \int_{0}^{s} \dfrac{ |s^{a'-a}-u^{a'-a}|}{(s-u)^{a'+1}} u^{a} u^{(n+1)a-na'}  du\right)\\
& \qquad \quad \qquad \times \left( \int_{0}^{t} \dfrac{ |t-v|^{b}+ |t^{b}-v^{b}|}{(t-v)^{b'+1}} dv \right)\\
&=\dfrac{c_3C_nd_5d^{(n)}_1}{ab\Gamma(a)\Gamma(b)}s^{(n+1)a-na'}t^{(n+1)b-nb'}.
\end{align*}
Hence 
\[s^{a-a'}t^{b-b'}|\tilde{L}^{2}_n(s, t)| \leq\dfrac{c_3C_nd_5d^{(n)}_1}{ab\Gamma(a)\Gamma(b)}s^{\gamma(n+1)}t^{\tilde{\gamma}(n+1)}.   \]
\noindent For $iii)$, we similarly have in view of (\ref{est:RF2}) 
\begin{align*}
&|P_n(s, v)-P_n(u, v)| = |I_{0+}^{a, b} f_n(s, v)-I_{0+}^{a, b} f_n(u, v) |\\
&\qquad \quad \quad \leq C_n \dfrac{s^{(n+2)a-na'}v^{(n+2)b-nb'}}{ab\Gamma(a)\Gamma(b)}v^{a}[|s-u|^b+|s^b-u^b|].
\end{align*}
Then
\begin{align*}
& |\tilde{L}^{3}_n(s, t)| \leq c_3 \int_{0}^{s}\int_{0}^{t} \dfrac{|t^{b'-b}-v^{b'-b}|s^{a'-a}|P_n(s, v)-P_n(u, v)|}{(s-u)^{a'+1}+(t-v)^{b'+1}} \, du\,  dv\\
&\leq \dfrac{c_3C_nd_6d^{(n)}_2}{ab\Gamma(a)\Gamma(b)}s^{(n+1)a-na'}t^{(n+1)b-nb'},
\end{align*}
and \[s^{a-a'}t^{b-b'}|\tilde{L}^{3}_n(s, t)| \leq\dfrac{c_3C_nd_6d^{(n)}_2}{ab\Gamma(a)\Gamma(b)}s^{\gamma(n+1)}t^{\tilde{\gamma}(n+1)}.   \]
This finishes the proof of the Claim \ref{claim: L'}.
\subsection{Asymptotic behaviour of $|C^{\star}_n|$:}
\begin{lem} \label{lemm: estimations}Let 
 \begin{eqnarray*}
\left\{\begin{array}{l}
p:= \min(1-a', 1-b'), \gamma:= a-a'+b-b',\\
\gamma_0:=\min( \gamma, a+b+p)\,\quad \mbox{and}\,\quad \eta:=\min(p, \gamma).
\end{array}\right.
\end{eqnarray*}
Then, as $n \to \infty$, 
\begin{enumerate}
\item[i)-] 
\begin{align}\label{asymp: d1, d2, l}
d^{(n)}_1=O(n^{a'-1}),\quad d^{(n)}_2=O(n^{b'-1}) \, \mbox{and}\quad l_n=O(n^{-p}) .
\end{align}
\item[ii)-] \[\kappa'_n=O(n^{-\gamma}),\quad \tilde{\kappa}'_n=O(n^{-(a+b)}), \, \quad \kappa_n=O(n^{-(a+b)}),\quad \mbox{and}\quad \tilde{\kappa}_n=O(n^{-\gamma}) .\]
\item[iii)-] \begin{align}\label{asymp: m}
&m_n=O(n^{-\gamma_0}),\qquad r_n=O(n^{-\eta})\qquad   \mbox{and}\quad C_n=O\left(\dfrac{B^n}{(n!)^{\eta}}\right) 
\end{align}
for some $B>0$.
\end{enumerate}
\end{lem}
\noindent \textbf{Proof of Lemma \ref{lemm: estimations}}:  Let us prove that $d^{(n)}_1=O(n^{a'-1})$. The proof of $d^{(n)}_2=O(n^{b'-1})$ is obtained by repeating the same arguments. We write
\[d^{(n)}_{1}= \int_{0}^{1/2}\dfrac{1-u^{a'-a}}{(1-u)^{a'+1}}u^{(n+2)a-na'}du+  \int_{1/2}^{1}\dfrac{1-u^{a'-a}}{(1-u)^{a'+1}}u^{(n+2)a-na'}du=:d^{(n)}_{1, 1}+d^{(n)}_{1, 2}.\]
For $d^{(n)}_{1, 1}$, we have 
\begin{align}\label{est:21}
&d^{(n)}_{1, 1} \leq 2^{-n(a-a')} \int_{0}^{1/2}\dfrac{1-u^{a'-a}}{(1-u)^{a'+1}}u^{2a}du=c2^{-n(a-a')}=O(n^{a'-1}),
\end{align}
since $n^{1-a'}2^{-n(a-a')} \longrightarrow 0$, as $ n \to \infty$ ($a>a'$).\\
For $d^{(n)}_{1, 2}$, by using the inequality $|u^{a'-a}-1|\leq 2^{1+a-a'}(1-u)$ valid for all $u \in [1/2, 1]$, 
we have
\[ d^{(n)}_{1, 2} \leq  (a-a')2^{1+ a-a'} \int_{0}^{1/2}\dfrac{ u^{n(a-a')+2a}}{(1-u)^{a'}} du \leq (a-a')2^{1+a-a'} \int_{0}^{1/2}\dfrac{ u^{n(a-a')}}{(1-u)^{a'}} du. \]
Now, since $ \ln(u) \leq u-1$ for all  $u \in [1/2, 1]$, we obtain
\begin{align}\label{est:22}
& d^{(n)}_{1, 2} \leq  (a-a')2^{1+a-a'} \int_{0}^{1/2}\dfrac{ e^{-n(a-a')(1-u)}}{(1-u)^{a'}} du \notag\\
&\leq (a-a')2^{1+a-a'} (a-a')^{a'-1}n^{a'-1} \int_{0}^{+\infty} t^{-a'}e^{-t}dt=c n^{a'-1} .
\end{align}
Therefore, the desired conclusion follows by combining (\ref{est:21}) and (\ref{est:22}).\\
Now, $l_n = O(n^{-p})$ is an immediate consequence of $d^{(n)}_1=O(n^{b'-1})$  and $d^{(n)}_2=O(n^{b'-1})$, together with the explicit expression of $l_n$. This finishes the proof of point $i)$.\\
For point $ii)$, we only prove $\kappa_n=O(n^{-(a+b)})$, since  the proof of the other  asymptotics is similar. Based on of the following Wendel's inequality (see \cite{W48}), valid for all  $0<s<1$  and $x>0$, 
\[\left( \dfrac{x}{x+s}\right)^{1-s} \leq \dfrac{\Gamma(x+s)}{x^{s}\Gamma(x)}\]
we get,  for all $0<s<1$  and  $x\geq 1$, 
\begin{align} \label{e:Wendel}
\dfrac{\Gamma(x)}{\Gamma(x+s)}\leq  x^{-s} \left( 1+ \dfrac{s}{x}\right)^{1-s} \leq x^{-s}(1+s)^{1-s}.
\end{align}
Then,  (\ref{e:Wendel}) implies 
\begin{align*}
& \kappa_n= \dfrac{\Gamma(\alpha_n)}{\Gamma(\alpha_n+ a)}\dfrac{\Gamma(\beta_n) }{\Gamma(\beta_n+b)}\leq (1+a)^{1-a}(1+b)^{1-b} \alpha_n^{-a}\beta_n^{-b}=O(n^{-a-b}).
\end{align*}
For point $iii)$, first recall that
\[m_n:= \kappa'_n + \tilde{\kappa'}_n(c_1d^{(n)}_1+ c_2d^{(n)}_2 +c_3d^{(n)}_1d^{(n)}_2)\]
and 
\[r_n:=(1+c_3d^{(n)}_1d^{(n)}_2)\tilde{\kappa}_n + (c_1d^{(n)}_1 +c_2d^{(n)}_2) \kappa_n +\dfrac{c_3(d_5d^{(n)}_1 +d_6d^{(n)}_2)}{ab\Gamma(a)\Gamma(b) }.\]
Then $m_n=O(n^{-\gamma_0})$ and $r_n=O(n^{-\eta})$ follow from points $i)$ and $ii)$. Now, let us  show that $C_n=O\left(\dfrac{B^n}{(n!)^{\eta}}\right) $ as $n\to \infty$. We have, for all $n\geq 1$, 
\[|C_n|=c_0\prod_{k=0}^{n-1}|r_k|. \]
Since $r_n= O(n^{-\eta})$ as $n \rightarrow\infty$, we deduce that there exist constants $M>0$,  an integer $N\geq 1$ and $C'_N>0$ such that for all $n\geq N$
\[|C_n|\leq C'_N \dfrac{M^n}{((n-1)!)^{\eta}},\]
which is smaller than $C'_N M^n/(n!)^{\eta}$ since $n^{\eta} \leq 2^{\eta}2^{n}$, for all $n \geq 1$. The proof of Lemma \ref{lemm: estimations} is now complete.
\begin{prop} \label{prop:C^*} We have, as $n\to\infty$, 
\begin{align}\label{e:estC*}
|C^{\star}_n|=O\left(\dfrac{B^n}{(n!)^{\eta}}\right),
\end{align}
 for some  $B>0$. In particular, $\sum_{n\geq 1} |C^{\star}_n|T^{n\gamma}<\infty$, where $\gamma= a-a'+b-b'$.
\end{prop}
\begin{Proof}
From  $C^{\star}_{n+1}=  m_nC^{\star}_{n} +l_nC_n $, we prove by induction that, for all $n \geq 1$,
 \begin{align} \label{e:expressionC*}
&C^{\star}_{n} = \left(\prod_{k=0}^{n-1}m_k\right)c^{\star}_0 + \sum_{j=0}^{n-1}\left(\prod_{k=j+1}^{n-1}m_k\right)g_j,
\end{align}
with $g_n:=l_nC_n$. Then 
 \begin{align} \label{est: seqC'}
&|C^{\star}_{n}| \leq  |c^{\star}_0|\left(\prod_{k=0}^{n-1}|m_k|\right) + \sum_{j=0}^{n-1}\left(\prod_{k=j+1}^{n-1}|m_k|\right)|g_j|.
\end{align}
First,  we get from (\ref{asymp: d1, d2, l}) that there exist $B>0$ and $N_0 \in \mathbb{N}$ such that for all $n \geq N_0$,
\begin{align}\label{est:g_n}
g_n:= l_nC_n \leq C_0 C'_{N_{0}} \dfrac{B^n}{n^p (n!)^{\eta}} .
\end{align}
In particular, the series $\sum_n |g_n|$ converges.\\
\noindent Second, using (\ref{asymp: m}) of Lemma \ref{lemm: estimations}, there exist $M>1$ and an integer $N \geq 2$, such that for all $k\geq N$ it holds $|m_k|\leq \dfrac{M}{ k^{\gamma_0}}$.
Then, on one hand for all $n \geq N+1$, we have
 \begin{align*} 
&\prod_{k=0}^{n-1}|m_k| = \left(\prod_{k=0}^{N-1}|m_k|\right)\left(\prod_{k=N}^{n-1}|m_k|\right)\leq C'_N M^{n-N}\left( \dfrac{(N-1)!}{(n-1)!}\right)^{\gamma_0}= C'_N\dfrac{M^{n}}{((n-1)!)^{\gamma_0}}.
\end{align*}
On the other hand, for $j\geq N-1$, we get 
 \begin{align*} 
&\prod_{k=j+1}^{n-1}|m_k|  \leq  M^{n-j-1}\left( \dfrac{j!}{(n-1)!}\right)^{\gamma_0}= \dfrac{M^n}{((n-1)!)^{\gamma_0}}\dfrac{(j!)^{\gamma_0}}{M^{j+1}} ,
\end{align*}
and for $j\leq  N-2$,
\begin{align*} 
&\prod_{k=j+1}^{n-1}|m_k| = \left(\prod_{k=j+1}^{N-1}|m_k|\right)\left(\prod_{k=N}^{n-1}|m_k|\right)\leq C'_N M^{n-N}\left( \dfrac{(N-1)!}{(n-1)!}\right)^{\gamma_0}=C'_N\dfrac{M^{n}}{((n-1)!)^{\gamma_0}}.
\end{align*}
Therefore (\ref{est: seqC'}) becomes, for all $n \geq N':=\max(N, N_0+1)$,
 \begin{align} \label{est: 2seqC'}
&|C^{\star}_{n}| \leq  |c^{\star}_0|C'_N\dfrac{M^{n}}{((n-1)!)^{\gamma_0}} +C'_N\dfrac{M^{n}}{((n-1)!)^{\gamma_0}}+ \dfrac{M^n}{((n-1)!)^{\gamma_0}} \left(\sum_{j=N'-1}^{n-1} \left(\dfrac{B}{M}\right)^j \dfrac{(j!)^{\gamma_0- \eta}}{j^p}\right).
\end{align} 
Notice that $\eta \leq \gamma_0$, so the series appearing on the right-hand side of  (\ref{est: 2seqC'}) may be divergent. Nevertheless, we have
{\small{
 \begin{align*} &\dfrac{1}{((n-1)!)^{\gamma_0}} \sum_{j=N'-1}^{n-1}\dfrac{1}{(j!)^p} \left(\dfrac{B}{M}\right)^j \dfrac{(j!)^{\gamma_0- \eta}}{j^p}=\dfrac{1}{((n-1)!)^{\eta}} \sum_{j=N'-1}^{n-1} \dfrac{1}{(j!)^p}\left(\dfrac{B}{M}\right)^j \left(\dfrac{j!}{(n-1)!}\right)^{\gamma_0- \eta}\\
 &\leq \dfrac{1}{((n-1)!)^{\eta}} \sum_{j=N'-1}^{n-1}\left(\dfrac{B}{M}\right)^j \leq \dfrac{c}{ ((n-1)!)^{\eta}}.
\end{align*} 
}}
Then, for all $n \geq N'$,
\begin{align} \label{est: 2seqC'}
&|C^{\star}_{n}| \leq  |c^{\star}_0|C'_N\dfrac{M^{n}}{((n-1)!)^{\gamma_0}} \leq |c^{\star}_0|C'_N2^{\gamma_0}\dfrac{(2M)^{n}}{(n!)^{\gamma_0}}, 
\end{align} 
since $n^{\gamma_0} \leq 2^{\gamma_0}2^{n}$, for all $n \geq 1$.  This completes the proof of  (\ref{e:estC*}). Furthermore, by (\ref{est: 2seqC'}) we get also  the convergence of the series $\sum_{n\geq 1} |C^{\star}_n|T^{n\gamma}.$ This concludes the proof of Proposition \ref{prop:C^*}. 
\end{Proof}
\begin{lem} \label{lem: strong}  Let $(\alpha, \beta)$ and $(\alpha', \beta')$ be  in $(0, 1/2]^{2}$ such that $(\alpha, \beta)\prec (\alpha^{\prime}, \beta^{\prime})  \preceq (\frac{1}{2}, \frac{1}{2})$  or $(\alpha^{\prime}, \beta^{\prime}) \prec(\alpha, \beta)  \preceq (\frac{1}{2}, \frac{1}{2})$. Let $1<\gamma < 1 +\min(\frac{1}{\alpha}, \frac{1}{\beta})$. Recall the expression of $\sigma^2(s, t)$ given by (\ref{e:sigma}). Then, there exists $0<\varepsilon<1$ such that for all $(s, t)\in (0, \varepsilon)^2$
\[ \left(\frac{3}{2}\right)^{1-\gamma} s^{\alpha (1- \gamma)}t^{\beta (1- \gamma)} \leq \sigma^{1- \gamma}(s, t) \leq \left(\frac{1}{2}\right)^{1-\gamma} s^{\alpha (1- \gamma)}t^{\beta (1- \gamma)}. \]
Consequently, $\int_{0}^T\int_{0}^T \sigma^{1- \gamma}(s, t) ds dt < \infty.$
\end{lem}
\begin{Proof} From (\ref{e: decompo1}) and by using Cauchy-Schwarz inequality we have for $z=(s, t)$,
\begin{align}
& \sigma(s, t) = \left\{ \int_{0}^{s} \int_{0}^{t} \left( K^{\alpha, \beta}((s, t), (u, v)) +  K^{\alpha', \beta'}((s, t), (u, v)) \right)^2 du dv\right\}^{1/2}\notag\\
&\leq   \left(\int_{0}^{s} \int_{0}^{t} (K^{\alpha, \beta})^2((s, t), (u, v)) du dv \right)^{1/2}  +  \left(\int_{0}^{s} \int_{0}^{t} (K^{\alpha', \beta'})^2((s, t), (u, v))  du dv \right)^{1/2}  \notag\\
&= \left(\int_{0}^{s}  K_{\alpha}^2(s,u) du   \right)^{1/2}  \left(\int_{0}^{t}  K_{\beta}^2(t,v) dv  \right)^{1/2}   + \left(\int_{0}^{s}  K_{\alpha'}^2(s,u) du   \right)^{1/2}  \left(\int_{0}^{t}  K_{\beta'}^2(t,v) dv  \right)^{1/2}  \notag\\
&=s^{\alpha}t^{\beta} + s^{\alpha'}t^{\beta'}= s^{\alpha}t^{\beta}(1 + s^{\alpha'- \alpha}t^{\beta'-\beta}).
\end{align}
Furthermore
\begin{align} & \quad \sigma(s, t) \geq \notag\\
&  \left|  \left(\int_{0}^{s} \int_{0}^{t} (K^{\alpha, \beta})^2((s, t), (u, v)) du dv \right)^{1/2} -  \left(\int_{0}^{s} \int_{0}^{t} (K^{\alpha, \beta'})^2((s, t), (u, v)) du dv \right)^{1/2}  \right|\notag\\
&= |s^{\alpha}t^{\beta} - s^{\alpha'}t^{\beta'}|= s^{\alpha}t^{\beta}|1 - s^{\alpha'- \alpha}t^{\beta'-\beta}|.
\end{align}
Taking into account that $\alpha < \alpha'$ and $\beta < \beta'$, the lemma follows by letting $s$ and $t$ tend to zero. 

\end{Proof}
\section{Appendix}
\subsection{Estimate for the Difference of Two-Parameter Rieman--Liouville Fractional integrals}
Let $0<\alpha, \beta <1$ and $s, t >0$. Let \( f: [0, s] \times [0, t]\longrightarrow \mathbb{R} \) be a measurable bounded function, with \( |f(u,v)| \leq M \). Assume that \( 0 < u \leq  s \), \( 0 < v \leq t \). Consider the two-parameter left-sided Riemann--Liouville fractional integral defined by:
\[
I_{0+}^{\alpha,\beta} f(x,y) := \frac{1}{\Gamma(\alpha)\Gamma(\beta)} \int_0^x \int_0^y (x - u)^{\alpha - 1}(y - v)^{\beta - 1} f(u,v) \, dv \, du.
\]

\begin{prop}\label{Prop:A1} For  \( 0 < x_2 \leq  x_1\leq s \), \( 0 < y_2 \leq y_1\leq t \), we have 
\[
\left| I_{0+}^{\alpha,\beta} f(x_1,y_1) - I_{0+}^{\alpha,\beta} f(x_2,y_2) \right|
\leq  \dfrac{2M(s^{\alpha }+t^{\beta})}{\alpha \beta \Gamma(\alpha)\Gamma(\beta)}  \mathcal{S}\left(x_1, x_2, y_1, y_2\right).
\]
with
\begin{equation}
\mathcal{S}\left(x_1, x_2, y_1, y_2\right):=\left|x_1^\alpha-x_2^\alpha\right|+\left|x_1-x_2\right|^\alpha+\left|y_1^\beta-y_2^\beta\right|+\left|y_1-y_2\right|^\beta,
\end{equation}

\end{prop}
\bprf  Set \[
\Delta := \left| I_{0+}^{\alpha,\beta} f(x_1, y_1) - I_{0+}^{\alpha,\beta} f(x_2, y_2) \right|.
\]
We write  \begin{align}\label{delta: decomp}
\Delta=\frac{1}{\Gamma(\alpha)\Gamma(\beta)}(A_0+A_1+A_2),
\end{align}
where
\begin{align*}
A_0&=\int_0^{x_2}\int_0^{y_2}\Bigl[(x_1-u)^{\alpha-1}(y_1-v)^{\beta-1}-(x_2-u)^{\alpha-1}(y_2-v)^{\beta-1}\Bigr]f(u,v)\,dv\,du,\\
A_1&=\int_0^{x_2}\int_{y_2}^{y_1}(x_1-u)^{\alpha-1}(y_1-v)^{\beta-1}f(u,v)\,dv\,du,\\
A_2&=\int_{x_2}^{x_1}\int_0^{y_1}(x_1-u)^{\alpha-1}(y_1-v)^{\beta-1}f(u,v)\,dv\,du.
\end{align*}
Hence
\begin{equation}\label{eq:tri}
|\Delta|\le \frac{1}{\Gamma(\alpha)\Gamma(\beta)}\bigl(|A_0|+|A_1|+|A_2|\bigr).
\end{equation}
At this stage, let us make the following useful and  trivial remark, that we will use later
\begin{remark}\label{Rem:A1} if $x_1=x_2$ or $y_1=y_2$, then $A_1=A_2=0$. 
\end{remark}
\begin{itemize}
    \item[$\bullet$]\textbf{Estimate of $A_0$:}  Since $|f|\leq M$, we have
    \begin{align*}
|A_0|
&\le M\Bigg[
\left(\int_0^{x_2}\bigl((x_2-u)^{\alpha-1}-(x_1-u)^{\alpha-1}\bigr)\,du\right)\left(\int_0^{y_2}(y_1-v)^{\beta-1}\,dv\right)\\
&\hspace{2.6cm}+
\left(\int_0^{x_2}(x_2-u)^{\alpha-1}\,du\right)\left(\int_0^{y_2}\bigl((y_2-v)^{\beta-1}-(y_1-v)^{\beta-1}\bigr)\,dv\right)
\Bigg].
\end{align*}
Using the elementary inequality $|ab-cd|\le |a-c||b|+|c||b-d|$, valid for all real numbers $a, b, c$ and $d$,  we get 
   \begin{align}\label{e:A_0}
|A_0|
&\le \dfrac{M}{\alpha \beta} \left\{ (x_1-x_2)^{\alpha}(y_1^{\beta}-(y_1-y_2)^{\beta})+ x_2^{\alpha}(y_2^{\alpha}-y_1^{\beta}+(y_1-y_2)^{\beta})\right\}.
\end{align}
Then 
   \begin{align}\label{e:A_0 prop}
|A_0|
&\le \dfrac{M}{\alpha \beta} (\mathcal{S} y_1^{\beta} + x_2^{\alpha} \mathcal{S}) \le \dfrac{M( s^{\alpha}+ t^{\beta})}{\alpha \beta} \mathcal{S}
\end{align}
Furthermore, from (\ref{e:A_0}) we deduce that
\begin{eqnarray}\label{e:A_0 modif}
\left\{\begin{array}{l}
|A_0|  \le \dfrac{M x_2^{\alpha}}{\alpha \beta}  \mathcal{S}\,\, \qquad \mbox{if } \, x_1=x_2 , \\
\\
|A_0|  \le \dfrac{M y_1^{\beta}}{\alpha \beta}  \mathcal{S}\,\, \qquad \mbox{if } \, y_1=y_2.
\end{array}\right.
\end{eqnarray}
\item[$\bullet$]\textbf{Estimate of $A_1$ and $A_2$:} Using again $|f|\leq M$, we have   
   \begin{align}\label{e:A_1 prop}
|A_1|
&\le M \left(\int_0^{x_2}(x_1-u)^{\alpha-1}\,du\right)\left(\int_{y_2}^{y_1}(y_1-v)^{\beta-1}\,dv\right)\notag \\
&=\dfrac{M}{\alpha \beta}(y_1-y_2)^{\beta} (x_1^{\alpha}-(x_1- x_2)^{\alpha}) \leq \dfrac{M x_1^{\alpha}}{\alpha \beta} \mathcal{S}\leq \dfrac{M s^{\alpha}}{\alpha \beta} \mathcal{S},
\end{align}
and 
  \begin{align}\label{e:A_2 prop}
|A_2|
&\le \dfrac{M}{\alpha \beta}(x_1-x_2)^{\alpha} y_1^{\beta} \leq \dfrac{M y_1^{\beta}}{\alpha \beta} \mathcal{S}\leq \dfrac{M t^{\beta}}{\alpha \beta} \mathcal{S}.
\end{align}
\end{itemize}
Plugging (\ref{e:A_0 prop}), (\ref{e:A_1 prop}) and (\ref{e:A_2 prop}) in (\ref{delta: decomp}) we get      \begin{align}\label{e:A_1 prop}
|\Delta|\le \frac{M}{\alpha \beta\Gamma(\alpha)\Gamma(\beta)} (2s^{\alpha} +2t^{\beta}) \mathcal{S}.
\end{align}
\nprf
Following the proof of Proposition \ref{Prop:A1}, and using (\ref{e:A_0 modif}) instead of (\ref{e:A_0 prop}) and by taking into account Remark \ref{Rem:A1}, we obtain with the notations of Proposition \ref{Prop:A1}:
\begin{corollary} \label{est:RF-d}
 For every $0\leq u\leq s$ and $0\leq v\leq t$, we have
   \begin{align} \label{est:RF1}
&\left| I_{0+}^{\alpha,\beta} f(u,t) - I_{0+}^{\alpha,\beta} f(u,v) \right| \leq \dfrac{Mu^{\alpha}}{\alpha \beta \Gamma(\alpha)\Gamma(\beta)}\left\{ |t-v|^{\beta}+ |t^{\beta} -v^{\beta}|\right\},
\end{align}
and
\begin{align}\label{est:RF2}
&\left| I_{0+}^{\alpha,\beta} f(s,v) - I_{0+}^{\alpha,\beta} f(u,v) \right| \leq \dfrac{Mv^{\beta}}{\alpha \beta \Gamma(\alpha)\Gamma(\beta)}\left\{ |s-u|^{\alpha}+ |s^{\alpha} -u^{\alpha}|\right\}.
\end{align}
\end{corollary}
\end{itemize}

\end{document}